\documentclass[a4paper]{amsart}

\usepackage{amssymb}
\usepackage{amsmath,amsthm}
\usepackage{enumerate,paralist}
\usepackage{amstext}
\usepackage{psfrag}
\usepackage{dsfont}
\usepackage[dvips]{epsfig}
\usepackage[english]{babel}
\usepackage{color}

\theoremstyle{plain}
\newtheorem{theorem}{Theorem}
\newtheorem{corollary}{Corollary}
\newtheorem{lemma}{Lemma}
\newtheorem{proposition}{Proposition}
\theoremstyle{definition}
\newtheorem{definition}{Definition}

\newtheorem{remark}{Remark}
\newtheorem{assumption}{Assumption}

\numberwithin{theorem}{section}
\numberwithin{corollary}{section}
\numberwithin{lemma}{section}
\numberwithin{definition}{section}
\numberwithin{example}{section}
\numberwithin{remark}{section}
\numberwithin{proposition}{section}
\numberwithin{assumption}{section}

\renewcommand{\leq}{\leqslant}
\renewcommand{\geq}{\geqslant}
\newcommand{\intl}{\int\limits}
\newcommand{\liml}{\lim\limits}
\newcommand{\suml}{\sum\limits}
\newcommand{\supl}{\sup\limits}
\newcommand{\prodl}{\prod\limits}

\newcommand{\cC}{\mathbb C}

\newcommand{\cR}{\mathbb R}

\newcommand{\cRd}{{{\mathbb R}^d}}

\newcommand{\Nat}{\mathbb N}

\newcommand{\PP}{\mathbb P}

\newcommand{\E}{\mathbb E}

\newcommand{\F}{\mathcal F}
\newcommand{\pd}{\partial}
\newcommand{\ffi}{\varphi}
\newcommand{\eps}{\varepsilon}

\newcommand{\sbs}{\subset}
\newcommand{\supp}{\mathop{\mathrm{supp}}\nolimits}

\newcommand{\id}{\mathop{\mathrm{Id}}\nolimits}
\newcommand{\Id}{\mathop{\mathrm{Id}}\nolimits}
\newcommand{\Dom}{\mathop{\mathrm{Dom}}\nolimits}

\newcommand{\Hess}{\mathop{\mathrm{Hess}}\nolimits}
\newcommand{\tr}{\mathop{\mathrm{tr}}\nolimits}
\newcommand{\diam}{\mathop{\mathrm{diam}}\nolimits}
\newcommand{\dist}{\mathop{\mathrm{dist}}\nolimits}

\usepackage{MnSymbol}

\begin{document}

\title[Approximations for killed  processes and   time-fractional equations]{Chernoff approximation for semigroups generated by killed Feller processes and Feynman formulae for time-fractional Fokker--Planck--Kolmogorov equations}

\author{ Yana A. Butko }
\date{\today}

\begin{abstract}
Semigroups, generated by Feller processes killed upon leaving a given domain, are considered. These semigroups correspond to  Cauchy--Di\-rich\-let type initial-exterior value problems in this domain for a class of evolution equations with (possibly non-local) operators. The considered semigroups are approximated by means of the Chernoff theorem.  For a class of killed Feller processes, the constructed Chernoff approximation converts into a Feynman formula, i.e. into a limit of $n$-fold iterated integrals of certain functions as $n\to\infty$. Representations of solutions of evolution equations by Feynman formulae can be used for direct calculations and simulation of underlying stochasstic processes. Further, a method to approximate solutions of time-fractional (including distributed order time-fractional) evolution equations is suggested. This method is based on connections between time-fractional and time-non-fractional evolution equations as well as on Chernoff approximations for the latter ones. Moreover, this method  leads to Feynman formulae for solutions of time-fractional evolution equations. To illustrate the method, a class of distributed order time-fractional diffusion equations is considered;  Feynman formulae  for solutions of the corresponding Cauchy and Cauchy--Dirichlet problems are obtained.

\bigskip

\end{abstract}

\maketitle

\tableofcontents

\section{Introduction}
Classical results of Analysis and Stochastics provide  deep connections between such mathematical objects as operator semigroups, evolution equations, path integrals and Markov processes.  A strongly continuous operator semigroup $\left(T_t\right)_{t\geq0}$ with  a given generator $L$ on a given Banach space $X$ allows to solve an initial (or  initial-boundary) value problem  for the corresponding evolution equation $\frac{\partial f}{\partial t}=Lf$ with the initial value $f_0$ by the identity $f(t)=T_tf_0$. On the other hand, if $(\xi_t)_{t\geq0}$ is a Markov process with generator $L$, the semigroup $\left(T_t\right)_{t\geq0}$ defines  transition probability $P(t,x,dy)$ of this   process  by the formula  $T_tf_0(x)=\int f_0(y)P(t,x,dy)=\mathbb{E}^x[f_0(\xi_t)]$. The last expression gives a stochastic representation for the solution of the above mentioned  evolution equation in $X$. A basic example is a $d$-dimensional Brownian motion  whose transitional probability $P(t,x,dy)$ is given by the Gaussian exponent  $P(t,x,dy):=(2\pi t)^{-d/2}\exp\left(   -\frac{|x-y|^2}{2t}\right)dy$, and  the corresponding evolution equation is the heat equation $\frac{\partial f}{\partial t}=\frac12\Delta f$ with the Laplace-operator $\Delta$. 

Usually, however, the semigroup $\left(T_t\right)_{t\geq0}$ or, respectively, the transition probability $P(t,x,dy)$ is not known explicitly. One of  analytical approaches   to  deal with unknown semigroups / transitional probabilities is  the use of  the Chernoff Theorem (which is a very broad generalization of the classical Daletskii--Lie--Trotter formula). This theorem provides conditions for a family (just a family, not a semigroup!) of bounded linear operators $(F(t))_{t\geq0}$ to  approximate the  semigroup $(T_t)_{t\geq0}$ with a given generator $L$  via the formula $T_t=\lim\limits_{n\to\infty}\left[  F(t/n) \right]^n$. This formula is called  \emph{Chernoff approximation of the semigroup $\left(T_t\right)_{t\geq0}$ by the family $(F(t))_{t\geq0}$}.
 The method of Chernoff approximation has the following advantage: if the family $(F(t))_{t\geq0}$ is given explicitly,  the expressions $\left[  F(t/n) \right]^n$ can be directly used for  calculations and hence for approximation of solutions of corresponding evolution equations, for computer modelling of considered dynamics, for approximation of transition probabilities of underlying Markov processes and hence for simulation of these processes.  Moreover, if all operators $F(t)$ are integral operators with elementary kernels (or pseudo-differential operators with elementary symbols), the identity  $T_t=\lim\limits_{n\to\infty}\left[  F(t/n) \right]^n$ leads to  representation of the semigroup $(T_t)_{t\geq0}$ by  limits of  $n$-fold iterated  integrals of  elementary functions when $n$ tends to infinity.  Such representations are called \emph{Feynman formulae}. Feynman formulae have an additional advantage. Namely, the limits in Feynman formulae   usually coincide with   functional (or path) integrals with respect to probability measures (i.e., with Feynman-Kac formulae providing stochastic representations for solutions of the corresponding evolution equations) or with respect to Feynman pseudomeasures (i.e., with Feynman path integrals).  Therefore,   representations of evolution semigroups by Feynman formulae  allow to establish new Feynman-Kac formulae and  give an additional tool to calculate path  integrals numerically. Note that different Feynman formulae for the same semigroup allow to establish  relations between different path integrals. This leads, in particular, to some  
 ``change-of-variables'' rules, connecting certain Feynman--Kac formulae and Feynman path integrals  (see the papers \cite{MR3455669,MR2863557}). 
 Chernoff approximation can be understood in some particular cases also as a construction of Markov chains approximating a given Markov process \cite{MR2502474} and as a numerical path integration method for solving corresponding PDEs/SDEs  \cite{Jakobsen}.

 It is worth to mention that the method of Chernoff approximation has a wide range of applications.  For example, this method has been used to investigate  {Schr\"{o}dinger type evolution equations}  in \cite{Remizov-JFA,MR2965550,MR2766564,MR2547389,MR2477291,MR2127488,MR2042203,MR1835462};  {stochastic Schr\"{o}dinger type equations} have been studied in \cite{MR3527030,MR2314121,MR2157588,MR2190085}.  
{Second order parabolic equations related to diffusions in different geometrical structures} (e.g., in  Eucliean spaces and their subdomains, Riemannian ma\-ni\-folds and  their subdomains, metric graphs, Hilbert spaces) have been studied, e.g.,  in \cite{MR3455669,MR3381153,MR3496753,technomag,MR3096384,MR3154319,MR2965551,MR2863557,MR2810757,MR2953994,MR2729591,
MR2536519,MR2483984,MR2423533,MR2359387,MR2449989,MR2314128,MR2133965,MR2114870,MR1991497}. 
{Evolution equations with non-local operators} generating some Markov processes in $\cRd$  have been considered in \cite{ChApprSubSem,MR3455669,MR2999096,MR2759261}. 
{Evolution equations with the Vladimirov operator} (this operator  is a $p$-adic analogue of the Laplace operator)  have been investigated in  \cite{MR2963683,MR2865952,MR2681357,MR2599557,MR2462100}. 
 Feynman formulae as a method of averaging of random Hamiltonians have been discussed in \cite{MR3588817,SakSmoOrl2014}. In the present note, the method of Chernoff approximation is applied to investigation of  initial-exterior value problems for evolution equations with non-local operators in a given subdomain $G$ of $\cRd$. Such problems correspond to  governing equations for Feller processes, killed upon leaving the domain $G$. For  a special class of killed Feller processes,   solutions of the corresponding initial-exterior problems are represented by Feynman formulae.

Fractional derivatives are natural extensions of their integer-order analogues  (see, e.g., \cite{MR1219954,MR1347689}). Evolution equations with partial derivatives of fractional order (fractional evolution equations), in particular, time- (and, possibly, space-) fractional diffusion equations (modelling anomalous diffusion)  have  been  applied to problems in physics, chemistry, biology, medicine, finance, hydrology and  other areas (see, e.g.,  \cite{MR2884383,MR1809268,MR2090004,MR1219954,MR1347689} and references therein). 
Many time-fractional evolution equations serve as governing  equations for stochastic processes. However, the processes, whose marginal density function  evolves in time according to a given time-fractional evolution equation, are usually non-Markovian (and hence, there is no semigroup structure behind the equation)  and are non-uniquely  defined by this marginal density function  (therefore, very different stochastic representations for a solution of a given fractional evolution equation are possible, see, e.g., \cite{MR3413862,MR2782245,MR2588003}).  The absence of the semigroup property for solutions of time-fractional evolution equations does not allow to apply the method of Chernoff approximation for such equations directly. Nevertheless,   several relations exist between time-fractional and ''standard'' (time-non-fractional) evolution equations: via a kind of subordination (see, e.g.,  \cite{MR2964432, MR1604710,MR1809268,MR1874479,MR2588003}) and via higher order operators (see, e.g., \cite{MR2027298,MR2489164,MR2491905,MR2965747,MR3400947}). These relations allow to construct some approximations for solutions of such time-fractional evolution equations  via Chernoff approximations for solutions of some related ''standard'' evolution equations.  In this note,  we present   approximations for solutions of (a class of)  time-fractional evolution equations using their connection to time-non-fractional equations via a kind of subordination.
In particular,  Feynman formulae are obtained for solutions of Cauchy and Cauchy--Dirichlet problems for distributed order time-fractional diffusion equations.

\section{Notations and preliminaries}

\subsection{Chernoff theorem and Feynman formulae}
In the sequel, the following version of the Chernoff theorem is used (cf. \cite{MR0231238,MR0417851}).
\begin{theorem}\label{0:thm:Chernoff-my}
Let $(F(t))_{t\geq0}$ be a family of bounded linear operators on a Banach space $X$. Assume that
\begin{enumerate}
\item[(i)]  $F(0)=\id$,

\item[(ii)] $\|F(t)\|\leq  e^{wt}  $ for some   $w\in\cR$  and all $t\geq0$,

\item[(iii)] the limit $L\ffi:=\liml_{t\to0}\frac{F(t)\ffi-\ffi}{t}$ exists for all $\ffi\in D$, where $D$ is a dense subspace in $X$ such that $(L,D)$ is closable and the closure $(L,\Dom(L))$ of $(L,D)$ generates a strongly continuous semigroup $(T_t)_{t\geq0}$.
\end{enumerate}
 Then the semigroup $(T_t)_{t\geq0}$ is  given by
\begin{equation}\label{0:eq:ChernoffFormula}
T_t\ffi=\liml_{n\to\infty}[F(t/n)]^n\ffi
\end{equation}
for all $\ffi\in X$, and the convergence is locally uniform with respect to  $t\geq0$. 
\end{theorem}
The formula \eqref{0:eq:ChernoffFormula} is  called \emph{Chernoff approximation of the semigroup $(T_t)_{t\geq0}$ by the family $(F(t))_{t\geq0}$}. Any family $(F(t))_{t\geq0}$, satisfying the assumptions (i)--(iii) of the Chernoff theorem~\ref{0:thm:Chernoff-my} with respect to a given semigroup $(T_t)_{t\geq0}$, is   called \emph{Chernoff equivalent} to this semigroup.
\begin{definition}
A \emph{Feynman formula} is a representation of a solution of an initial (or initial-boundary) value problem for an evolution equation (or, equivalently, a representation of the semigroup solving the problem) by a limit of $n$-fold iterated integrals of some functions as $n\to\infty$.
\end{definition}

One should not confuse the notions of Chernoff approximation and Feynman formula.
On the one hand, not all Chernoff approximations can be directly interpreted as Feynman formulae since, generally, the operators $(F(t))_{t\geq0}$ do not have to be neither integral operators, nor pseudo-differential operators.  On the other hand,   representations of solutions of evolution equations in the form of Feynman formulae can be obtained by different methods, not necessarily via the Chernoff Theorem. And such Feynman formulae may have no relations to any  Chernoff approximation, or their relations may be quite indirect (see, e.g., Feynman formulae~\eqref{eq:FF:CDP-nonLoc}, \eqref{FF:fracCP}, \eqref{FF:fracCDP} below).

\subsection{Killed Feller processes and their generators}
\label{Section:CDP:1}

In this Subsection, we follow the exposition of~\cite{MR3156646} and \cite{BLM}.  Let $Q$ be a locally compact separable metric space. Let $C_0(Q)$  be the  set  of all bounded continuous functions $\ffi$ on $Q$ such that   $\forall\,\eps>0\,\,\exists\,\text{ a compact }\, K_\ffi^\eps\subset Q\,\text{ with }\, |\ffi(q)|<\eps\,\,\forall\,q\notin K_\ffi^\eps$. The set $X:=C_0(Q)$ is a  Banach space  endowed with the supremum-norm.
 A semigroup of bounded linear operators $(T_t)_{t\geq0}$ on the Banach space $X$ is called \emph{Feller semigroup} if it is a strongly continuous semigroup, it is \emph{positivity preserving} (i.e. $T_t\ffi\geq0$ for all $\ffi\in X$ with $\ffi\geq0$) and it is \emph{sub-Markovian} (i.e. $T_t\ffi\leq1$ for all $\ffi\in X$ with $\ffi\leq1$).  Its generator (called \emph{Feller generator}) is defined by 
$\Dom(L):=\left\{ \ffi\in X \,\big|\,\,\lim_{t\to0}\frac{T_t\ffi-\ffi}{t}\quad \text{ exists in }  X    \right\}$,
$L\ffi:=\lim_{t\to0}\frac{T_t\ffi-\ffi}{t}\quad\forall\ffi\in\Dom(L)$.
Let $(\Omega, \mathcal{F}, \mathbb{P})$ be a probability space with a filtration $(\F_t)_{t\geq0}$, and let  $(\xi_t)_{t\geq0}$ be a temporally homogeneous Markov process with state space $(Q,\mathcal{B}(Q))$, where $\mathcal{B}(Q)$ is the Borel sigma-algebra of $Q$. The process $(\xi_t)_{t\geq0}$ is called \emph{Feller process} if its transition semigroup $(T_t)_{t\geq0}$,
$
T_t\ffi(x):=\int_Q \ffi(y) \mathbb{P}^x(\xi_t\in dy)\equiv\mathbb{E}^x[\ffi(\xi_t)],
$
 is a Feller semigroup.   
Using a version of the Riesz representation theorem, one can extend a Feller  semigroup $(T_t)_{t\geq0}$ to a sub-Markovian semigroup on the space   $B_b(Q)$ of bounded Borel functions. 
A sub-Markovian semigroup $(T_t)_{t\geq0}$ (on $B_b(Q)$) is called \emph{strong Feller semigroup} if $T_t\,:\,B_b(Q)\to C_b(Q)$ for all $t>0$.
Note that a strong Feller semigroup need not  be Feller and vice versa. 
 Some conditions on Feller semigroups to be strong Feller can be found in Lemma~1.12, Theorem~1.14 and Theorem~1.15 of \cite{MR3156646}.
The  process $(\xi_t)_{t\geq0}$ (resp., the  semigroup $(T_t)_{t\geq0}$ with $T_t\ffi(x):=\mathbb{E}^x[\ffi(\xi_t)]$) is called \emph{doubly Feller} if it is both Feller and strong Feller.

Let $Q=\cRd$ and $X=C_0(\cRd)$. Let  $(\xi_t)_{t\geq0}$ be a Feller process on $\cRd$, $(T_t)_{t\geq0}$ be the corresponding Feller semigroup  and $(L,\Dom(L))$ be its Feller generator. We define the \emph{pointwise generator} $(L_p,\Dom(L_p))$ of $(T_t)_{t\geq0}$ by
\begin{align}\label{CDP:1:eq:L_p}
&\Dom(L_p):=\left\{\ffi\in X\,\,\big|\,\, \exists\,g\in X\,:\,\liml_{t\to0}\frac{T_t\ffi(x)-\ffi(x)}{t}=g(x)\,\,\forall\,x\in\cRd  \right\},\nonumber\\
&
L_p\ffi(x):=\liml_{t\to0}\frac{T_t\ffi(x)-\ffi(x)}{t}=g(x)\quad\forall\,\ffi\in\Dom(L_p),\,\,x\in\cRd.
\end{align}
Then $(L,\Dom(L))=(L_p,\Dom(L_p))$ by Theorem~1.33 of \cite{MR3156646}. Moreover,  if\footnote{The assumption $C^\infty_c(\cRd)\subset\Dom(L)$ is quite standard and holds in many cases, see, e.g., \cite{MR3156646}.} $C^\infty_c(\cRd)\subset\Dom(L)$,  then  we have also  $C^2_0(\cRd):=\{\ffi\in C^2(\cRd)\,:\,\pd^\alpha\ffi\in C_0(\cRd),\,|\alpha|\leq2 \}\subset\Dom(L)$. And  $L\ffi(x)$ is given for each $\ffi\in C^2_0(\cRd)$ and each $x\in\cRd$  by the folowing formula:
\begin{equation}\label{2:eq:g35} 
\begin{split}
    L\ffi(x)
    &=
    -C(x)\ffi(x) - B(x)\cdot\nabla \ffi(x)
    + \tr(A(x)\Hess\ffi(x))\\
    &\phantom{==}+ \int_{y\neq 0} \left( \ffi(x+y) - \ffi(x)
    - \frac{y\cdot\nabla \ffi(x)}{1+|y|^2}\right)\,N(x,dy),
\end{split}
\end{equation}
where  $\Hess\ffi$ is the Hessian matrix of  second order partial derivatives of $\ffi$; as well as  $C(x)\geq0$,  $B(x)\in\cRd$, $A(x)\in\mathbb{R}^{d\times d}$ is a symmetric positive semidefinite matrix    and $N(x,\cdot)$ is a Radon measure  on $\cRd\setminus \{0\}$ with  $\int_{y\neq 0}
|y|^2(1+|y|^2)^{-1}\,N(x,dy)<\infty$  for each $x\in\cRd$.
Therefore,  $L$ is an integro-differential operator on $C^2_0(\cRd)$. And this operator is non-local if $N\neq0$. 
This class of generators  $L$ includes, in particular,  fractional Laplacians $L=-(-\Delta)^{\alpha/2}$ and relativistic Hamiltonians $\sqrt[\alpha]{(-\Delta)^{\alpha/2}+m(x)}$, $\alpha\in(0,2)$, $m>0$.

Let now  $G\subset\cRd$ be a bounded domain (connected open set) and let $Y:=C_0(G)$ be the set of all continuous 
 functions on $G$ that tend to zero as $x\in G$ approaches the boundary $\pd G$.  Then $Y$ is a Banach space with the supremum norm $\|\cdot\|_Y$, $\|\ffi\|_Y:=\sup_{x\in G}|\ffi(x)|$.  For a Feller process $(\xi_t)_{t\geq0}$ on $\cRd$, we define the first exit time from $G$ 
 by
$$
\tau_G:=\inf\{t>0\,:\,\xi_t\notin G\}.
$$
Let $(\xi_t^o)_{t\geq0}$ denote the killed process on $G$, i.e.,
$$
\xi_t^o=\left\{
\begin{array}{ll}
\xi_t, & t<\tau_G,\\
\pd, & t\geq\tau_G,
\end{array}\right.
$$
where $\pd$ denotes a cemetery point (i.e. an isolated point $\pd\notin G$ added to the state space $G$). We say that a boundary point $x\in\pd G$ is \emph{regular} for $G$ if $\mathbb{P}^x(\tau_G=0)=1$. We say that $G$ is \emph{regular} if every point $x\in\pd G$ is regular.\footnote{Due to Theorem~2.2. of \cite{MR1473631}, if  a boundary point $x\in\pd G$ satisfies the external cone condition, then it is regular for the case when $(\xi_t)_{t\geq0}$ is  symmetric $\alpha$-stable. In particular, any Lipschitz domain is regular in this case.  }  Let  $(\xi_t)_{t\geq0}$ be a doubly Feller process on $\cRd$ and $G$ be regular. Then $(T^o_t)_{t\geq0}$, such that
$$
T^o_t\ffi(x):=\mathbb{E}^x\left[\ffi\left(\xi^o_t\right)\right],\quad \ffi\in Y,\,\, x\in G,
$$
is a Feller semigroup on $Y$ (cf. Lemma~2.2 in \cite{BLM}). Let $(L_o,\Dom(L_o))$ be the Feller generator of $(T^o_t)_{t\geq0}$ on $Y$. The operator $(L_o,\Dom(L_o))$ is described in Proposition~\ref{CDP:1:prop:BLM} below (cf. Thm.~2.3 and Lemma~2.6 in~\cite{BLM}).  Note that each element $\ffi\in Y$ can be extended by zero outside $G$; this extension is again denoted by $\ffi$ and belongs to the space $X$.

\begin{proposition}\label{CDP:1:prop:BLM}
Let  $(\xi_t)_{t\geq0}$ be a doubly Feller process on $\cRd$, $(T_t)_{t\geq0}$ be the corresponding Feller semigroup and $(L,\Dom(L))$ be its Feller generator. Let  $G\subset\cRd$ be a bounded regular domain. The generator of the killed Feller process  $(\xi_t^o)_{t\geq0}$ is characterized as follows
\begin{enumerate}
\item[(i)] $\Dom(L_o)=\{\ffi\in Y\,\,:\,\, L_p\ffi\in Y\}$, where  $L_p$ is given by~\eqref{CDP:1:eq:L_p},  $L_p\ffi\in Y$ means that $L_p$  is applied to the zero extension of $\ffi$ on $\cRd$, $L_p\ffi(x)$ exists for each $x\in G$ and the function $[x\mapsto L_p\ffi(x)]$ belongs to $Y$. Moreover, it holds that  $L_o\ffi(x)=L_p\ffi(x)$ for all $\ffi\in\Dom(L_o)$ and the limit in~\eqref{CDP:1:eq:L_p}
exists 
locally uniformly on $G$ (i.e., uniformly with respect to $x\in K$ for each compact $K\subset G$).

\item[(ii)] Assume that $C^\infty_c(\cRd)\subset\Dom(L)$.  Then, for each $\ffi\in C_0(\cRd)\cap C^2(G)$, we have
$$
L_p\ffi(x)=L\ffi(x),\quad\forall \, x\in G,
$$
where $L$ is the integro-differential operator given by the formula~\eqref{2:eq:g35}.
\end{enumerate}
\end{proposition}

\begin{remark}\label{rem:ACP-CDP}
(i) The abstract Cauchy problem in $Y$ for the  evolution equation $\frac{d f}{dt}=L_o f$ with an initial condition $f_0\in\Dom(L_o)$ can be interpreted as  
the following Cauchy--Dirichlet type initial--exterior value problem\footnote{Such problems are discussed, e.g., in \cite{MR3318251}, see also \cite{MR1383011} for the stationary case.}: 
\begin{align}\label{CDP:1:eq:IEP}
&\frac{\partial f}{\partial t}(t,x)=Lf(t,x),\quad t>0,\, x\in G,\nonumber\\
&f(0,x)=f_0(x),\quad x\in G,\\ 
&f(t,x)=0,\quad t>0,\,\quad x\in \cRd\setminus G.\nonumber 
\end{align}
And the function $f(t,x):=T^o_tf_0(x)$, extended by zero outside $G$, solves the problem for each $f_0\in\Dom(L_o)$ (e.g., by Theorem~4.1.3 in \cite{MR710486}).

(ii) Let, additionally,  the operator $(L,\Dom(L))$ in $X$ be a \emph{local operator outside $\overline{G}:=G\cup\pd G$}, i.e. for each $x\in\cRd\setminus\overline{G}$  and each $\ffi_1$, $\ffi_2\in\Dom(L)$ such that $\ffi_1$ and $\ffi_2$ coincide on $\overline{G}$ and on some neighbourhood of $x$, one has $L\ffi_1(x)=L\ffi_2(x)$. 
For example, consider the integro-differential operator $L$ given by~\eqref{2:eq:g35} with $N(x,dy)$  such that 
\begin{align}\label{CDP:2:eq:CensoredN}
N(x,dy)=\left\{
\begin{array}{ll}
1_{0<|y|\leq\dist(x,\pd G)}(y)N(x,dy), & x\in G,\\
0, & x\in\cRd\setminus G.
\end{array}
\right.
\end{align}
The integral part of such operator $L$ gives rise to the so-called \emph{censored processes} in $G$, cf.~\cite{MR2006232,MR3461088}.
If $L$ is local outside $\overline{G}$, the abstract Cauchy problem in $Y$ for the  evolution equation $\frac{d f}{dt}=L_o f$ can be interpreted as the following Cauchy--Dirichlet problem: 
\begin{align}\label{CDP:1:eq:CDP}
&\frac{\partial f}{\partial t}(t,x)=Lf(t,x),\quad t>0,\, x\in G,\nonumber\\
&f(0,x)=f_0(x),\quad x\in G,\\ 
&f(t,x)=0,\quad t>0,\,\quad x\in \partial G.\nonumber 
\end{align}
And again the function $f(t,x):=T^o_tf_0(x)$ solves  this problem for each $f_0\in\Dom(L_o)$.


\end{remark}

\subsection{Distributed-order fractional derivatives}
There exist many different notions of fractional derivatives. We discuss only two versions of them. One defines the \emph{Caputo} (or \emph{Caputo-Dzhrbashyan}) \emph{fractional derivative} of  order $\beta$, $\beta\in(0,1)$, for a (sufficiently good) function $u$ by
\begin{align*}
\frac{\pd^\beta}{\pd t^\beta} u(t):=\frac{1}{\Gamma(1-\beta)}\intl_0^t \frac{u'(r)}{(t-r)^\beta}dr,
\end{align*} 
where $\Gamma$ is the Euler's Gamma-function.
Let $U$ be the Laplace transform of $u$, i.e. $U(s):=\int_0^\infty e^{-st}u(t)dt$.  Then the Laplace transform of the Caputo derivative  $\frac{\pd^\beta}{\pd t^\beta} u$ of $u$ can be calculated as follows:
\begin{align*}
\intl_0^\infty e^{-st}\frac{\pd^\beta}{\pd t^\beta} u(t) dt= s^{\beta}U(s)-s^{\beta-1}u(0+).
\end{align*}
The \emph{Riemann--Liouville fractional derivative} of  order $\beta$, $\beta\in(0,1)$, for a (sufficiently good) function $u$ is defined by
\begin{align*}
\left(\frac{d}{dt}\right)^\beta u(t):=\frac{1}{\Gamma(1-\beta)}\frac{d}{dt}\intl_0^t \frac{u(r)}{(t-r)^\beta}dr.
\end{align*} 
Then the Laplace transform of the Riemann-Liouville derivative  $\left(\frac{d}{dt}\right)^\beta u$ of $u$ can be calculated  as follows:
\begin{align*}
\intl_0^\infty e^{-st}\left(\frac{d}{dt}\right)^\beta u(t) dt= s^{\beta}U(s).
\end{align*}
Comparing both Laplace transforms and taking into account that the Laplace transform of $t^{-\beta}$ is $s^{\beta-1}\Gamma(1-\beta)$, one sees that if $u$ is absolutely continuous on bounded intervals then the Riemann-Liouville and Caputo derivatives of $u$ are related by
\begin{align}\label{Appl:3:eq:Cap-RL}
\frac{\pd^\beta}{\pd t^\beta} u(t)=\left(\frac{d}{dt}\right)^\beta u(t)-\frac{t^{-\beta}u(0+)}{\Gamma(1-\beta)}.
\end{align}
The Riemann-Liouville fractional derivative is more general since it does not require the first derivative of $u$ to exist.  Therefore, one may adopt the right hand side of the formula~\eqref{Appl:3:eq:Cap-RL} to define the Caputo derivative.
 The further generalization is to consider the so-called \emph{distributed order fractional derivative} $\mathcal{D}^\mu$ with the  order $\mu$ determined by a finite Borel measure $\mu$ defined on the interval $(0,1)$ and  such that $\mu(0,1)>0$ (cf.~\cite{MR2376152,MR2463064,MR2251542,MR2208034}):
\begin{align*}
\mathcal{D}^\mu u(t):=\intl_0^1 \frac{\pd^\beta}{\pd t^\beta} u(t)\mu(d\beta)=\intl_0^1\left[\left(\frac{d}{dt}\right)^\beta u(t)-\frac{t^{-\beta}u(0)}{\Gamma(1-\beta)}\right]\mu(d\beta).
\end{align*}
If $\mu$ is  the Dirac delta-measure  $\delta_{\beta_0}$ concentrated at a point $\beta_0\in(0,1)$, we return to Caputo fractional derivative of $\beta_0-$th order.


\section{Chernoff approximation of semigroups ge\-ne\-ra\-ted by some killed Feller processes}
\label{Section:CDP:2}

\subsection{General approach}

Let   $X= C_0(\cRd)$.  
Let  $(\xi_t)_{t\geq0}$ be a doubly Feller process on $\cRd$, $(T_t)_{t\geq0}$ be the corresponding (doubly Feller) semigroup and $(L,\Dom(L))$ be its Feller generator.
Let a  family  $(F(t))_{t\geq0}$ of bounded linear operators on  $X$ be Chernoff equivalent\footnote{Note, that some families $(F(t))_{t\geq0}$,  which are Chernoff equivalent to certain Feller semigroups, are constructed, e.g., in  \cite{ChApprSubSem,MR2999096,MR2759261}.} to the semigroup  $(T_t)_{t\geq0}$. Therefore,  $\|F(t)\|\leq e^{kt}$  for some  $k\in \cR$  and all  $t\geq 0$, and there exists a core $D$ for the operator $(L,\Dom(L))$ such that $\lim_{t\to0}\left\|\frac{F(t)\ffi-\ffi}{t}-L\varphi\right\|_X=0$ for all  $\varphi\in D$. Let us fix this core $D$.
 Let  $G\subset\cRd$ be a bounded regular domain. Let $(T^o_t)_{t\geq0}$ be the strongly continuous semigroup on $Y:=C_0(G)$ generated by the killed Feller process $(\xi^o_t)_{t\geq0}$ on $G$. Let $(L_o,\Dom(L_o))$ be the Feller generator of $(T^o_t)_{t\geq0}$. Our aim is to construct a family $(F_o(t))_{t\geq0}$ of bounded linear operators on $Y$ which is Chernoff equivalent to $(T^o_t)_{t\geq0}$. The family $(F_o(t))_{t\geq0}$ will be constructed by a proper modification of the family $(F(t))_{t\geq0}$  which is Chernoff equivalent to the semigroup  $(T_t)_{t\geq0}$ on $X$. 
 To this aim, we need some preparations.

\begin{assumption}\label{CDP:2:ass:E+D_0}
 We assume that there exists a set $D_o\sbs \Dom(L_o)\cap C^2_b(G)$ and a mapping $\mathcal{E}: Y\to C_c(\cRd)\sbs X$ such that
\begin{enumerate}
\item[(i)] $D_o$ is a core for $L_o$;
\item[(ii)]  $\mathcal{E}(\ffi)\big|_{\overline{G}}=\ffi$ for all $\ffi\in Y$;
\item[(iii)]  the mapping $\mathcal{E}$ is linear;
\item[(iv)] the mapping $\mathcal{E}$ preserves the supremum norm, i.e. $\|\ffi\|_Y=\|\mathcal{E}(\ffi)\|_X$ for all $\ffi\in Y$;
\item[(v)] $\mathcal{E}:\, D_o\to D$, where $D$ is a fixed core for $(L,\Dom(L))$;
\item[(vi)] $L(\mathcal{E}(\ffi))(x)=L_o\ffi(x)$  for each $\ffi\in D_o$ and each $x\in G$.
\end{enumerate}
\end{assumption}

\begin{remark}\label{rem:Lunardi}
 The space $Y$ can be naturally embedded into $X$ by assigning to  each $\ffi\in Y$  zero values outside the domain $G$. However, such embedding produces from smooth functions in $G$ only continuous functions in $\cRd$. This may violate the requirement (v) of Assumption~\ref{CDP:2:ass:E+D_0}. Note that $\Dom(L_o)$ typically contains functions whose zero extensions do not belong to  $\Dom(L)$. Moreover, there is no reason to expect the existence of  a core $D_o$ such that the zero extensions of its elements belong to  $\Dom(L)$. In particular, the sets of sufficiently smooth functions with compact supports in $G$  can not serve as a core even for the Laplacian $\Delta$ in $Y$\footnote{This fact together with its proof have been communicated to the author by Professor Alessandra Lunardi.}!  Indeed, assume that there exists a core $D_o\subset C_c(G)$ for $(\Delta,\Dom(\Delta))$ in $Y$.  Then for each $\ffi\in\Dom(\Delta)$ there exists a sequence $(\ffi_n)_{n\in\Nat}\subset D_o$ such that $\|\ffi_n-\ffi\|_Y\to0$ and $\|\Delta\ffi_n-\Delta\ffi\|_Y\to0$ as $n\to\infty$. By Corollary~3.1.21~(i),~(ii) and Remark 2.1.5 in \cite{MR3012216}, $\Dom(\Delta)$ is continuously embedded in $C^1(\overline{G})$. Hence $\left\|\frac{\pd\ffi_n}{\pd x_i}-\frac{\pd\ffi}{\pd x_i}\right\|_Y\to0$ as $n\to\infty$ for all $i=1,\ldots,d$. Therefore, $\nabla\ffi\big|_{\pd G}=0$ for each $\ffi\in\Dom(\Delta)$. This is however wrong since, e.g., the function $\ffi(x):=\sin x$ belongs to $\Dom(\Delta)=\Dom\left(\frac{d^2}{dx^2}\right)$ for $G:=(0,\pi)$ and $\frac{d\ffi}{dx}(x)=\cos x$ is not equal to zero on $\pd G$.
\end{remark}

\begin{remark}\label{CDP:2:rem:L_ABC-i}
Let the generator $L$ of a doubly Feller semigroup $(T_t)_{t\geq0}$  be given for each $\ffi\in C^2_0(\cRd)$ by  
\begin{align}\label{1:3:eq:L_ABC}
L\ffi(x)=-C(x)\ffi(x) - B(x)\cdot\nabla \ffi(x)+ \tr(A(x)\Hess\ffi(x)),
\end{align}
where the coefficients $A$, $B$ and $C$ are of the class\footnote{ Here and in the sequel, $C_b(\cRd)$ stands for the space of bounded continuous functions on $\cRd$; $C_c(\mathbb{R}^d)$ stands for the space of continuous functions on $\cRd$ with compact support;
$C^{0,\alpha}(\mathbb{R}^d)$ stands for H\"{o}lder continuous functions on $\cRd$ with exponent $\alpha\in(0,1]$; 
$C^m_b(\mathbb{R}^d)=\left\{\ffi\in C^m(\cRd)\,:\,\pd^\beta\ffi\in C_b(\cRd),\,|\beta|\leq m    \right\}$; 
$C^{m,\alpha}(\mathbb{R}^d)=\left\{\ffi\in C^m(\cRd)\,:\,\pd^\beta\ffi\in C^{0,\alpha}(\cRd),\,|\beta|= m    \right\}$; $C^{m,\alpha}_b(\mathbb{R}^d)=C^{m,\alpha}(\mathbb{R}^d)\cap C^m_b(\mathbb{R}^d)$;
$C^{m,\alpha}_c(\mathbb{R}^d)=C^{m,\alpha}(\mathbb{R}^d)\cap C_c(\mathbb{R}^d)$.} $C^{2,\alpha}_b(\cRd)$ for some $\alpha\in(0,1)$. Let there exist $a_0$, $A_0\in\cR$ such that 
\begin{align}\label{eq:CP:uniformEllipt}
0<a_0\leq A_0<\infty\quad\text{ and }\quad a_0|z|^2\leq z\cdot A(x)z\leq A_0|z|^2\quad\text{ for all }\,\,x,z\in\cRd.
\end{align}
 Assume that the coefficients $A$, $B$, $C$ are such that the set $C^{2,\alpha}_c(\cRd)$ is a core for the generator $L$ in $X$, and let $D:=C^{2,\alpha}_c(\cRd)$. Consider $D_o:=\left\{\ffi\in C^{2,\alpha}(G)\,:\, \ffi, L\ffi\in Y\right\}.$ If the boundary $\pd G$ is of the class $C^{4,\alpha}$ then there exists a strongly continuous semigroup  $(T^o_t)_{t\geq0}$ on $Y$ generated by the closure  of $(L,D_o)$ in $Y$ and  Assumption~\ref{CDP:2:ass:E+D_0} is fulfilled due to Thm.~2.2 and Thm.~3.4 in~\cite{MR2860750}.
\end{remark}


\begin{remark}\label{CDP:2:rem:L_ABC-ii}
   Let now the generator $L$ of a  Feller semigroup $(T_t)_{t\geq0}$ on $X$ be such that $C^\infty_c(\cRd)\subset\Dom(L)$ and $L\ffi$ is given for each $\ffi\in C^2_0(\cRd)$   by  formula~\eqref{2:eq:g35}.
Consider $L$ as the sum $L:=L_1+L_2$, where $L_1$ is the differential operator given by  formula~\eqref{1:3:eq:L_ABC} and $L_2$ is  the integral part
$$
L_2\ffi(x):=\intl_{y\neq0}\left(\ffi(x+y)-\ffi(x)-\frac{y\cdot\nabla\ffi(x)}{1+|y|^2}\right)N(x,dy),\quad\forall\ffi\in C^2_0(\cRd),\,\,\forall\,x\in\cRd.
$$
Let  
the coefficients $A$, $B$ and $C$ be again  of the class $C^{2,\alpha}_b(\cRd)$ for some $\alpha\in(0,1)$. And let there exist $a_0$, $A_0\in\cR$ such that \eqref{eq:CP:uniformEllipt} holds. Assume that   the closure of $(L_1,C^{2,\alpha}_c(\cRd))$ generates a strongly continuous contraction semigroup on $X$.
Choose the core  $D:=C^{2,\alpha}_c(\cRd)\sbs X$  for the generator $L_1$. Let  the boundary $\pd G$ be of the class $C^{4,\alpha}$ for some $\alpha\in(0,1)$.   Then, as in Remark~\ref{CDP:2:rem:L_ABC-i}, there exists an extension $\mathcal{E}$ satisfying Assumption~\ref{CDP:2:ass:E+D_0}~(ii)-(vi) with respect to $L_1$. Let $U\subset\cRd$ be another bounded domain such that $G\subset U$. One may consider, e.g., $U\equiv U_\eps:=\{x\in\cRd\,:\,\dist(x,G)<\eps\}$ for some small constant $\eps>0$.  Multiplying the extension $\mathcal{E}$ with a proper cut-off function, one obtains another extension $\mathcal{E}_U$, satisfying Assumption~\ref{CDP:2:ass:E+D_0}~(ii)-(vi) with respect to $L_1$ and the condition $\mathcal{E}_U: Y\to C_c(U)$.  Assume that $N(x,dy)$ is such that $(L_2,D)$ is $L_1$-bounded\footnote{Note that the fractional Laplacian $-(-\Delta)^{\alpha/2}$ is $\Delta$-bounded for all $\alpha\in(0,2)$.} and the closure of $(L,D)$, $L=L_1+L_2$, generates a doubly Feller semigroup on $X$. Let there exist a core $D_o\sbs \{\ffi\in C^{2,\alpha}(G)\,:\,\ffi,\,L_1\ffi\in Y\}$ for the corresponding killed generator $(L_o,\Dom(L_o))$. Then the extension $\mathcal{E}_U$  satisfies  Assumption~\ref{CDP:2:ass:E+D_0}~(ii)-(v) with respect to $L$, $D$ and $D_o$.  Since $L$ is a non-local operator, Assumption~\ref{CDP:2:ass:E+D_0}~(vi) is not fulfilled automatically.  And, for each $\ffi\in D_o$ and each $x\in G$, we have
 \begin{align*}
 L(\mathcal{E}_U(\ffi))(x)-L_o\ffi(x)&=\intl_{y\neq0}\left(\mathcal{E}_U(\ffi)(x+y)-\ffi(x+y)\right)N(x,dy)\\
 &
 =\intl_{y\in\left(-x+U\setminus\overline{G}\right)}\mathcal{E}_U(\ffi)(x+y)N(x,dy).
 \end{align*}
 Let $L_2$ satisfy the additional condition: there exists a bounded domain $U\supset G$ such that for each $x\in G$ holds
 \begin{align}\label{CDP:2:eq:NforE}
  \intl_{y\in\left(-x+U\setminus\overline{G}\right)}N(x,dy)=0.
  \end{align} 
Then the extension $\mathcal{E}_U$  satisfies  Assumption~\ref{CDP:2:ass:E+D_0}. The condition~\eqref{CDP:2:eq:NforE} actually means that the process $(\xi_t)_{t\geq0}$ is allowed to leave the domain $G$ either continuously, or by a sufficiently large jump which brings the process even out of $U$. Note that if $N(x,dy)$  corresponds to censored processes (i.e., $N(x,dy)$ satisfies~\eqref{CDP:2:eq:CensoredN}), then the condition~\eqref{CDP:2:eq:NforE} is fulfilled. The condition~\eqref{CDP:2:eq:NforE} is also fulfilled if, e.g., $\supp N(x,\cdot)\subset\cRd\setminus K$ for all $x\in G$ and some compact $K$ such that $\cup_{x\in G}\left(-x+U\setminus\overline{G}\right)\subset K$. One can take as $K$, e.g., a ball $B_R(x_0)$ such that its center $x_0\in G$ and its raduius $R> 2\diam U$. 
\end{remark}

Consider now a  continuous  monotone  function $s: (0, \infty) \to (0, \infty)$  such that 
$$
\liml_{t\to0}\frac{s(t)}{t}=0.
$$
 Define the set   $G_{s(t)} \subset G$  by
 $$
 G_{s(t)}:=\{x\in G \,:\, \dist(x, \partial G)>s(t)\}.
 $$
Let  $\left(\phi_{s(t)}\right)_{t>0}$ be a family of  functions  $\phi_{s(t)}: \cRd \to [0,1]$   such that all $\phi_{s(t)}\in C^\infty_c(G)$,  we have  $\phi_{s(t)}(x)=1$,  $\forall\,x\in G_{s(t)}$,   $\forall\,t>0$,
and  $\lim_{\,t\to t^*}\|\phi_{s(t)}-\phi_{s(t^*)}\|_X=0$ for each  $t^*>0$.
Note that   functions  $\phi_{s(t)}$ converge poitwise to the indicator $1_G$ of the domain  $G$ when $t\to 0$. 
Consider the family $(F_o(t))_{t\geq0}$  of operators on $Y$ defined by  $F_o(0):=\Id$ and for each  $t> 0$,  each $\ffi\in Y$  and each $x\in G$
\begin{equation}\label{1:eq:4:Fo}
F_o(t)\ffi(x):= \phi_{s(t)}(x)[F(t)\mathcal{E}(\ffi)](x)
\end{equation}
where the given  family  $(F(t))_{t\geq0}$  is Chernoff equivalent to the semigroup  $(T_t)_{t\geq0}$ on $X$ generated by $(L,\Dom(L))$ and $F'(0)\varphi=L\varphi$ for all  $\varphi\in D$.

\begin{lemma}\label{1:4:le:F_o}
The family $(F_o(t))_{t\geq0}$ acts on $Y$ and $\|F_o(t)\|\leq\|F(t)\|$. If the family $(F(t))_{t\geq0}$ is strongly continuous on $X$ then the family $(F_o(t))_{t\geq0}$ is strongly continuous on $Y$.
\end{lemma}

\begin{proof}
The family $(F_o(t))_{t\geq0}$ acts on $Y$ since, if  $\ffi\in Y=C_0(G)$, then  $\mathcal{E}(\ffi)\in C_c(\mathbb{R}^d)\sbs C_\infty(\mathbb{R}^d)=X$, $F(t)\mathcal{E}(\ffi)\in X$ and $\phi_{s(t)}[F(t)\mathcal{E}(\ffi)]\in  Y$. Moreover,
\begin{align*}
\|F_0(t)\ffi\|_Y&=\supl_{x\in G}|\phi_{s(t)}(x)[F(t)\mathcal{E}(\ffi)](x)|\\
&
\leq \|F(t)\mathcal{E}(\ffi)\|_X\leq\|F(t)\|\|\mathcal{E}(\ffi)\|_X=\|F(t)\|\|\ffi\|_Y.
\end{align*}

Let us show the strong continuity of the family $(F_o(t))_{t\geq0}$ under assumption that  the family $(F(t))_{t\geq0}$ is strongly continuous on $X$.  First, for each  $\ffi\in Y$
\begin{align*}
\liml_{t\to0}\|F_o(t)&\ffi-\ffi\|_Y=\liml_{t\to0}\supl_{\,x\in G}\big|\phi_{s(t)}(x)[F(t)\mathcal{E}(\ffi)](x)-\ffi(x)\big|\\
&
= \liml_{t\to0}\supl_{\,x\in G}\big|\phi_{s(t)}(x)\big([F(t)\mathcal{E}(\ffi)](x)-\mathcal{E}(\ffi)(x)\big)+\ffi(x)[\phi_{s(t)}(x)-1]\big|\\
&
\leq \liml_{t\to0}\|F(t)\mathcal{E}(\varphi)-\mathcal{E}(\varphi)\|_X +\liml_{t\to0}\supl_{\,x\,\in\, \overline{G\setminus G_{s(t)}}}|\varphi(x)|\\
&
=0
\end{align*}
due to strong continuity at zero of the family $(F(t))_{t\geq0}$  on $X$  and uniform continuity of  $\ffi$ on the compact  $\overline{G\setminus G_{s(t)}}$. Second, for each  $t^*>0$ and each  $\ffi\in Y$
\begin{align*}
&\liml_{t\to t^*}\|F_o(t)\ffi-F_o(t^*)\ffi\|_Y\\
&=\liml_{t\to t^*}\supl_{\,x\in G}\big|\phi_{s(t)}(x)[F(t)\mathcal{E}(\ffi)](x)-\phi_{s(t^*)}(x)[F(t^*)\mathcal{E}(\ffi)](x)\big|\\
&
=\liml_{t\to t^*}\supl_{\,x\in G}\bigg|\phi_{s(t)}(x)\big([F(t)\mathcal{E}(\ffi)](x)-[F(t^*)\mathcal{E}(\ffi)](x)\big)+\\
&\phantom{qwwqwqwfgjhgjhgdfghfsqqwqwqwq}+(\phi_{s(t)}(x)-\phi_{s(t^*)}(x))[F(t^*)\mathcal{E}(\ffi)](x)\bigg|\\
&
\leq \liml_{t\to t^*}\|\phi_{s(t)}\|_Y\cdot\|[F(t)\mathcal{E}(\ffi)]-[F(t^*)\mathcal{E}(\ffi)]\|_X 
+\|F(t^*)E(\ffi)\|_X\cdot\|\phi_{s(t)}-\phi_{s(t^*)}\|_Y\\
&
=0
\end{align*}
due to strong continuity  of the family $(F(t))_{t\geq0}$  on $X$  and  propersties of the family $\left(\phi_{s(t)}\right)_{t>0}$. Hence Lemma is proved.
\end{proof}

\begin{theorem}\label{CDP:2:thm}
Under Assumption~\ref{CDP:2:ass:E+D_0},  the family $(F_o(t))_{t\ge0}$ is Chernoff equivalent to  the semigroup $(T^o_t)_{t\geq0}$, i.e.
$$
T^o_t\ffi=\liml_{n\to\infty}\left[F_o(t/n)   \right]^n\ffi
$$
for each $\ffi\in Y$ locally uniformly with respect to $t\geq0$.
\end{theorem}

\begin{proof}
Due to Lemma~\ref{1:4:le:F_o}, we have $\|F_o(t)\|\leq\|F(t)\|\leq e^{kt}$ for some $k\in\cR$ and all $t\geq0$.  Hence it is sufficient to show that $
\lim_{t\to 0}\| t^{-1}(F_o(t)\varphi -\varphi) - L_o\varphi\|_Y=0
$ for all $\ffi\in D_o$.  Due to Assumption~\ref{CDP:2:ass:E+D_0}
\begin{align*}
&\bigg\| \frac{F_o(t)\ffi -\ffi}{t} - L_o\ffi\bigg\|_Y=\supl_{x\in G}\bigg| \frac{\phi_{s(t)}(x)[F(t)\mathcal{E}(\ffi)](x) -\ffi(x)}{t} - L_o\ffi(x)\bigg|\\
&
\leq \supl_{x\in G}\bigg[|\phi_{s(t)}(x)|\bigg| \frac{F(t)\mathcal{E}(\ffi)(x) -\mathcal{E}(\ffi)(x)}{t} - L\mathcal{E}(\ffi)(x)\bigg|\\
&
\phantom{fjvndfkvfvjvbjfvbjsbfvsbvs}   +\big(| \ffi(x)/t  |+ |L_o\ffi(x)|\big)|1- \phi_{s(t)}(x)| \bigg]\\
&
\leq \left\|\frac{F(t)\mathcal{E}(\ffi) -\mathcal{E}(\ffi)}{t} - L\mathcal{E}(\ffi)\right\|_X+\supl_{x\in \overline{G\setminus G_{s(t)}}}\big(| \ffi(x)/t  |+ |L_o\ffi(x)|\big)\\
&
\to 0,\quad\text{ as }\,\, t\to0.
\end{align*}
Indeed, $\lim_{t\to0}\left\|\frac{F(t)\mathcal{E}(\ffi) -\mathcal{E}(\ffi)}{t} - L\mathcal{E}(\ffi)\right\|_X=0$ since 
$\mathcal{E}(\ffi)\in D$ and $F'(0)=L$ on $D$ by our assumptions. Further, $\ffi\in D_o\sbs C^2_b(G)\cap Y$. Hence $\ffi$ is Lipschitz on $\overline{G}$, i.e. there exists a constant $M>0$ such that  the inequality  $|\ffi(x)-\ffi(z)|\leq M|x-z|$ holds for  all $x$, $z\in \overline{G}$. Moreover,  for each $x\in G\setminus G_{s(t)}$, there exists  at least one point $z_x\in\pd G$ such that $\dist(x,z_x)\leq s(t)$. Therefore, $|\ffi(x)|=|\ffi(x)-\ffi(z_x)|\leq M s(t)$ for each $x\in G\setminus G_{s(t)}$.  And 
$$
\liml_{t\to0}\supl_{x\in \overline{G\setminus G_{s(t)}}}\frac{|\ffi(x)|}{t}\leq\liml_{t\to0}M\frac{s(t)}{t}=0.
$$
Besides, since $\ffi\in D_o\sbs \Dom(L_o)$, we have $L_o\ffi\in Y=C_0(G)$. Hence 
$$
\lim_{t\to0}\supl_{x\in G\setminus G_{s(t)}}|L_o\ffi(x)|=0
$$
 due to uniform continuity of the function $L_o\ffi$ on compacts $\overline{G\setminus G_{s(t)}}$. Thus, Theorem is proved.
\end{proof}

\begin{remark}
Analogues of Theorem~\ref{CDP:2:thm} are also valid in unbounded domains $G\sbs\cRd$, in domains $G$ of a locally compact metric space $Q$ and in other couples of Banach spaces $X$ and $Y$ (e.g., $X:=L^p(\cRd)$ and $Y:=L^p(G)$, $p\in[1,\infty)$) under  corresponding  modifications of Assumption~\ref{CDP:2:ass:E+D_0} and   properties of the family $\left(\phi_{s(t)}\right)_{t>0}$, as well as under additional assumption on the existence of the semigroup $(T^o_t)_{t\geq0}$.
\end{remark}

\subsection{Feynman Formula for semigroups generated by a class of killed Feller processes}\label{Subsection:CDP-nonLoc}

Let $(T_t)_{t\geq0}$ be a doubly Feller semigroup on $X$ whose generator $(L,\Dom(L))$ is such that  the set $D:=C^\infty_c(\cRd)$ is a core for $L$. Hence $L\ffi$ is given by formula~\eqref{2:eq:g35} for each $\ffi\in C^2_0(\cRd)$. Assume that the coefficients $A$, $B$, $C$ in formula~\eqref{2:eq:g35} are bounded and continuous. Let  there exist $a_0$, $A_0\in\cR$ such that condition~\eqref{eq:CP:uniformEllipt} holds.
 Let the measure $N(x,dy)$ in  formula~\eqref{2:eq:g35} do not depend on $x$, i.e. $N(x,dy):=N(dy)$ for all $x\in\cRd$. Let $(\eta_t)_{t\geq0}$ be the convolution semigroup on $\cRd$ corresponding\footnote{I.e.  the Fourier transforms $\mathcal{F}[\eta_t]$ of sub-probability measures $\eta_t$ for all $t\geq0$ are given by $\mathcal{F}[\eta_t](x)=(2\pi)^{-d/2} e^{-tr(x)}$, where the function $r\,:\,\cRd\to\cC$ is defined by
$
r(x):=\int_{\cRd\setminus\{0\}}\left(1-e^{iy\cdot x}+\frac{iy\cdot x}{1+|y|^2}\right)N(dy).
$} to $N(dy)$.
 Then, by Thm.~3.1 in \cite{MR2999096} and Remark~15  in \cite{MR3455669}, the following family $(F(t))_{t\geq0}$ on $X$ is Chernoff equivalent to $(T_t)_{t\geq0}$: $F(0)=\Id$ and for all $t>0$, all $\ffi\in X$ and all $x\in\cRd$
\begin{align}\label{eq:F(t)-NonLocCP}
F(t)\ffi(x):=\frac{e^{-tC(x)}}{\sqrt{(4\pi t)^{d}\det A(x)}}\intl_{\cRd}\intl_{\cRd}
e^{-\frac{A^{-1}(x)(z-x+tB(x)+y)\cdot(z-x+tB(x)+y)}{4t}}\ffi(y)dy\,\eta_t(dz).
\end{align}
Moreover, the family $(F(t))_{t\geq0}$ is a strongly continuous family of contractions. Note also that, for $g(x)\equiv1$, we have $F(t)g(x)=\exp\left\{-tC(x)\right\}\leq1$ for all $x\in\cRd$.

Let $G\subset\cRd$ be a regular bounded domain. Consider the corresponding Feller semigroup $(T^o_t)_{t\geq0}$ on $Y$. Let Assumption~\ref{CDP:2:ass:E+D_0} be fulfilled for some core $D_o$ of the generator of $(T^o_t)_{t\geq0}$ and for some extension $\mathcal{E}\,:\,Y\to X$ with respect to $D_o$ and $D:=C^\infty_c(\cRd)$. Then, by Theorem~\ref{CDP:2:thm}, the family $(F_o(t))_{t\geq0}$, constructed from the family $(F(t))_{t\geq0}$ in~\eqref{eq:F(t)-NonLocCP} through the formula~\eqref{1:eq:4:Fo}, is Chernoff equivalent to the semigroup $(T^o_t)_{t\geq0}$. Hence $F_o(0)=\Id$ and for all $t>0$ and all $\ffi\in Y$
\begin{align}\label{eq:Fo(t)-NonLocCP}
&F_o(t)\ffi(x):=\frac{\phi_{s(t)}(x)e^{-tC(x)}}{\sqrt{(4\pi t)^{d}\det A(x)}}\times\\
&
\times\intl_{\cRd}\intl_{\cRd}
\exp\left\{-\frac{A^{-1}(x)(z-x+tB(x)+y)\cdot(z-x+tB(x)+y)}{4t}\right\}\mathcal{E}(\ffi)(y)dy\,\eta_t(dz).\nonumber
\end{align}
Therefore, we have uniformly with respect to $x_0\in G$ and  uniformly with respect to $t\in(0,t^*]$ for all $t^*>0$
\begin{align}\label{eq:CPnonLoc-Fno(t)}
T^o_t\ffi(x_0)=\liml_{n\to\infty} F^n_o(t/n)\ffi(x_0)=&\liml_{n\to\infty}\intl_{\cRd}\intl_{\cRd}\ldots\intl_{\cRd}\intl_{\cRd}\left(\prodl_{k=1}^n \phi_{s(t/n)}(x_{k-1})\right)\mathcal{E}(\ffi)(x_n)\times\nonumber\\
&
\times\Psi^{x_0}_{t,n}(x_1,\ldots,z_n)\,dx_n\eta_{t/n}(dz_n)\cdots dx_1\eta_{t/n}(dz_1),
\end{align}
where
\begin{align*}
&\Psi^{x_0}_{t,n}(x_1,\ldots,z_n):=\left(\prodl_{k=1}^n(4\pi t/n)^{-d/2}(\det A(x_{k-1}))^{-1/2} \right)
\exp\left\{-\frac{t}{n}\suml_{k=1}^n C(x_{k-1})\right\}\times\\
&
\times\exp\left\{-\suml_{k=1}^n\frac{A^{-1}(x_{k-1})(z_k-x_{k-1}+tB(x_{k-1})+x_k)\cdot(z_k-x_{k-1}+tB(x_{k-1})+x_k)}{4t/n}\right\}.
\end{align*}
Since all  $\phi_{s(t)}$ are smooth functions with compact supports in $G$  and $\mathcal{E}(\ffi)$ is a continuous function with compact support $K:=\supp \mathcal{E}(\ffi)$,  the $2n$-fold iterated integrals over $\cRd$ in~\eqref{eq:CPnonLoc-Fno(t)} coincide with the following $2n$-fold multiple integral
\begin{align*}
\Phi^\ffi_n(t,x_0):=\intl_{G^{n-1}\times K\times\cR^{nd}}&\left(\prodl_{k=1}^n \phi_{s(t/n)}(x_{k-1})\right)\mathcal{E}(\ffi)(x_n)\times\\
&
\times\Psi^{x_0}_{t,n}(x_1,\ldots,z_n)\,dx_1\cdots dx_n\eta_{t/n}(dz_1)\cdots \eta_{t/n}(dz_n).
\end{align*}
Consider also
\begin{align*}
\Theta^\ffi_n(t,x_0):=\intl_{G^{n}\times\cR^{nd}}\ffi(x_n)\Psi^{x_0}_{t,n}(x_1,\ldots,z_n)\,dx_1\cdots dx_n\eta_{t/n}(dz_1)\cdots \eta_{t/n}(dz_n).
\end{align*}
Let us show that for all $t>0$ and all $x_0\in G$ holds
\begin{align}\label{eq:CDnonLOc-justEq}
T^o_t\ffi(x_0)=\liml_{n\to\infty}\Theta^\ffi_n(t,x_0).
\end{align}
And the convergence in~\eqref{eq:CDnonLOc-justEq} is locally uniform with respect to $x_0\in G$ and uniform with respect to $t\in(0,t^*]$ for all $t^*>0$. So, consider a number $t^*>0$ and a compact  $\Upsilon\subset G$.   Let $x_0\in\Upsilon$ and $t\in(0,t^*]$. Then
\begin{align*}
&\left|\Phi^\ffi_n(t,x_0)- \Theta^\ffi_n(t,x_0)\right|
\leq\intl_{G^{n-1}\times K\times\cR^{nd}} \left(\prodl_{k=1}^n |\phi_{s(t/n)}(x_{k-1})-1_G(x_{k-1})|\right)|\mathcal{E}(\ffi)(x_n)|\times\\
&
\phantom{tralalatralalatralalatralalatrala} \times\Psi^{x_0}_{t,n}(x_1,\ldots,z_n)dx_1\cdots dx_{n}\eta_{t/n}(dz_1)\cdots \eta_{t/n}(dz_n)\\
&
+ \intl_{ G^{n-1}\times\cR^{nd}}\left(\,\intl_{K\setminus {G}} |\mathcal{E}(\ffi)(x_n)|\Psi^{x_0}_{t,n}(x_1,\ldots,z_n)dx_n\right)dx_1\cdots dx_{n-1}\eta_{t/n}(dz_1)\cdots \eta_{t/n}(dz_n).
\end{align*}
Let us estimate each of the summands separately. Denote the first summand by $I^\ffi_n(t,x_0)$ and the second by $J^\ffi_n(t,x_0)$.  We have with $g(x)\equiv1$
\begin{align*}
&I^\ffi_n(t,x_0)\\
&
\leq\|\ffi\|_Y|\phi_{s(t/n)}(x_{0})-1_G(x_{0})|\intl_{\cRd}\ldots\intl_{\cRd}\Psi^{x_0}_{t,n}(x_1,\ldots,z_n)dx_n\eta_{t/n}(dz_n)\cdots dx_{1}\eta_{t/n}(dz_1)\\
&
=\|\ffi\|_Y|\phi_{s(t/n)}(x_{0})-1_G(x_{0})| \left(F^n(t/n)g(x_0)\right)\leq \|\ffi\|_Y|\phi_{s(t/n)}(x_{0})-1_G(x_{0})|.
\end{align*}
By the construction of the sets $G_{s(t)}$, there exists $N\in\Nat$ such that $\Upsilon\subset G_{s(t^*/N)}$. Hence for all $n\geq N$, $x_0\in\Upsilon$, $t\in(0,t^*]$ holds $|\phi_{s(t/n)}(x_{0})-1_G(x_{0})|=0$. Therefore,
\begin{align*}
\liml_{n\to\infty}I^\ffi_n(t,x_0)=0\quad\text{ uniformly with respect to }\, x_0\in \Upsilon,\,\, t\in(0,t^*].
\end{align*}
Consider now the second summand $J^\ffi_n(t,x_0)$.  Due to condition~\eqref{eq:CP:uniformEllipt}, we  have  for all  $x\in\overline{G}$, $y\in K$, $z\in\cRd$, $t\in(0,t^*]$ and $n\in\Nat$
\begin{align*}
&\frac{e^{-tC(x)/n}}{\sqrt{(4\pi t/n)^{d}\det A(x)}}
\exp\left\{-\frac{A^{-1}(x)(z-x+tB(x)/n+y)\cdot(z-x+tB(x)/n+y)}{4t/n}\right\}\\
&
\leq M(A_0/a_0)^{d/2} (4A_0\pi t/n)^{-d/2}\exp\left\{-\frac{|x-y|^2}{4A_0t/n} \right\},
\end{align*}
where
$$
M:=\supl_{x\in\overline{G},\,y\in K,\, z\in\cRd,\,t\in(0,t^*], n\in\Nat }\exp\left\{-\frac{|z+tB(x)/n|^2+2(z+tB(x)/n)\cdot(y-x)}{4A_0t/n}  \right\}<\infty.
$$
Therefore, with $c:=M(A_0/a_0)^{d/2}$ and $p_{A_0}(t,x,y):=(4A_0\pi t)^{-d/2}\exp\left\{-\frac{|x-y|^2}{4A_0t} \right\}$
\begin{align*}
J^\ffi_n(t,x_0)&\leq c\left(\supl_{x\in \overline{G}} \intl_{K\setminus {G}} p_{A_0}(t/n,x,y)|\mathcal{E}(\ffi)(y)|\,dy\right)\left( F^{n-1}(t/n)g(x_0)\right)\\
&
\leq c\left(\supl_{x\in \overline{G}} \intl_{K\setminus {G}} p_{A_0}(t/n,x,y)|\mathcal{E}(\ffi)(y)|\,dy\right).
\end{align*}
Denote by $G^\delta$ the $\delta$-neighborhood of $G$ in $\cRd$, i.e. $G^\delta:=\{x\in\cRd\,:\,\dist(x,G)<\delta\}$, $\delta>0$. Fix any $\eps>0$.  Since  $\mathcal{E}(\ffi)$ is a continuous function on $\cRd$ which equals zero on $\pd G$, there exists $\delta>0$ such that $|\mathcal{E}(\ffi)|\leq\eps/2$ on $G^\delta\setminus G$. Hence
\begin{align*}
&\supl_{x\in \overline{G}} \intl_{K\setminus {G}} p_{A_0}(t/n,x,y)|\mathcal{E}(\ffi)(y)|\,dy\leq \supl_{x\in \overline{G}} \intl_{G^\delta\setminus {G}} p_{A_0}(t/n,x,y)|\mathcal{E}(\ffi)(y)|\,dy\\
&
+\supl_{x\in \overline{G}} \intl_{K\setminus {G^\delta}} p_{A_0}(t/n,x,y)|\mathcal{E}(\ffi)(y)|\,dy\leq\frac{\eps}{2}+\|\ffi\|_Y\supl_{x\in \overline{G}} \intl_{K\setminus {G^\delta}} p_{A_0}(t/n,x,y)\,dy.
\end{align*}
Due to Gaussian fall off of $p_{A_0}$ there exists $N\in\Nat$ such that for all $n\geq N$ and all $t\in(0,t^*]$  holds
$$
\|\ffi\|_Y\supl_{x\in \overline{G}} \intl_{K\setminus {G^\delta}} p_{A_0}(t/n,x,y)\,dy\leq\frac{\eps}{2}.
$$
Consequently, since $\eps>0$ has been chosen arbitrary, 
\begin{align*}
\liml_{n\to\infty}J^\ffi_n(t,x_0)=0\quad\text{ uniformly with respect to }\, x_0\in \Upsilon,\,\, t\in(0,t^*].
\end{align*}
Therefore, the following statement  is proved.
\begin{proposition}\label{prop:CDP-nonLoc}
Under all assumptions of this Subsection, the following Feynman formula holds for the semigroup $(T^o_t)_{t\geq0}$:
\begin{align}\label{eq:FF:CDP-nonLoc}
&T^o_t\ffi(x_0)=\liml_{n\to\infty}\Theta^\ffi_n(t,x_0)
=\liml_{n\to\infty}\intl_{G^n\times\cR^{nd}}\ffi(x_n)\left(\prodl_{k=1}^n(4\pi t/n)^{-d/2}(\det A(x_{k-1}))^{-1/2} \right)\times\nonumber\\
&
\times
\exp\left\{-\suml_{k=1}^n\frac{A^{-1}(x_{k-1})(z_k-x_{k-1}+tB(x_{k-1})+x_k)\cdot(z_k-x_{k-1}+tB(x_{k-1})+x_k)}{4t/n}\right\}\times\nonumber\\
&
\phantom{tralala}\times\exp\left\{-\frac{t}{n}\suml_{k=1}^n C(x_{k-1})\right\}dx_1\cdots dx_n\eta_{t/n}(dz_1)\cdots \eta_{t/n}(dz_n),\quad\forall\,\ffi\in Y,\,\,\forall\,x_0\in G.
\end{align}
And the convergence in this  Feynman formula is locally uniform with respect to $x_0\in G$ and uniform with respect to $t\in(0,t^*]$ for all $t^*>0$.
\end{proposition}
As a particular case, we have the following (cf.~\cite{MR2729591}).
\begin{corollary}\label{cor:example:CDP-BGS}
Let all the assumptions of this Subsection be fulfilled. Let $N(dy)\equiv0$, i.e. $\eta_t=\delta_0$ for all $t\geq0$.  Then   the family $(F(t))_{t\geq0}$ given in~\eqref{eq:F(t)-NonLocCP} has the following view:   $F(0):=\Id$ and for all $t>0$ and all $\ffi\in X$
  \begin{align}\label{1:eq:F^ABC+}
F(t)\ffi(x):&=\frac{e^{-tC(x)}}{\sqrt{(4\pi t)^{d}\det A(x)}}\intl_{\cRd}
e^{-\frac{A^{-1}(x)(x-tB(x)-y)\cdot(x-tB(x)-y)}{4t}}\ffi(y)dy\\
&
\equiv e^{-tC(x)}\intl_{\cRd}e^{\frac{A^{-1}(x)B(x)\cdot(x-y)}{2}}e^{-t\frac{|A^{-1/2}(x)B(x)|^2}{4}}\ffi(y)p_A(t,x,y)dy,\nonumber
\end{align}
 where for all $x,y\in\cRd$
 \begin{equation}\label{1:eq:3:p_A}
p_A(t,x,y) :=\frac{1}{\sqrt{(4\pi t)^{d}\det A(x)}} \exp\bigg(-\frac{A^{-1}(x)(x-y)\cdot(x-y)}{4t}\bigg).
\end{equation}
 This   family $(F(t))_{t\geq0}$ is a strongly continuous family of contractions on $X$ which is Chernoff equivalent to the semigroup $(T_t)_{t\geq0}$.  Moreover, the corresponding semigroup $(T^o_t)_{t\geq0}$ can be approximated via the following Feynman formula:
\begin{align}\label{CDP:3:FF}
T^o_t\varphi(x_0)=& \liml_{n\to\infty} \intl_{ G^n}
\exp\bigg(-\frac{t}{n}\suml_{j=1}^n \left(C(x_{j-1})+\frac14\left|A^{-1/2}(x_{j-1})B(x_{j-1})\right|^2   \right)\bigg) \nonumber\\
&
\quad\times\exp\bigg(-\frac12\suml_{j=1}^n   A^{-1}(x_{j-1})B(x_{j-1})\cdot(x_j-x_{j-1})\bigg) \varphi(x_n)\nonumber \\
 &
\quad\times p_A(t/n,x_0,x_1) \cdots p_A(t/n, x_{n-1},x_n) dx_1\dots dx_n,
\end{align}
for each $\ffi\in Y$, each $x_0 \in G$ and each $t > 0$. The convergence in the Feynman formula~\eqref{CDP:3:FF} is locally uniform with respect to $x_0\in G$ and uniform with respect to  $t\in(0,t^*]$ for all $t^*>0$.
\end{corollary}
\begin{remark}
It is assumed in this Subsection that $C^\infty_c(\cRd)$ is a core for $L$. If $L$ is given by formula~\eqref{1:3:eq:L_ABC} with continuous and bounded coefficients $A$, $B$, $C$, then $C^\infty_c(\cRd)\subset C^{2,\alpha}_c(\cRd)\subset\Dom(L)$ and hence $C^{2,\alpha}_c(\cRd)$ is also a core for $L$. Therefore, one may consider $D:=C^{2,\alpha}_c(\cRd)$ (or any other bigger core for $L$) and find corresponding $D_o$ and $\mathcal{E}$ such that Assumption~\ref{CDP:2:ass:E+D_0} holds. In principle, the bigger $D$ is chosen, the easier is to find  $\mathcal{E}$  satisfying the condition~(v) of Assumption~\ref{CDP:2:ass:E+D_0} that $\mathcal{E}(D_o)\subset D$.  As it has been shown in \cite{MR2729591}, the family  $(F(t))_{t\geq0}$ given in~\eqref{1:eq:F^ABC+} satisfies the condition
$$
\lim_{t\to0}\left\|\frac{F(t)\ffi-\ffi}{t}-L\ffi\right\|_X=0\quad\text{ for all }\,\,\ffi\in C^{2,\alpha}_c(\cRd),\,\alpha\in(0,1).
$$
Therefore, it is sufficient  to assume that $C^{2,\alpha}_c(\cRd)$ is a core for $L$ in Corollary~\ref{cor:example:CDP-BGS}.
\end{remark}

\section{Approximation of solutions and  Feynman formulae for  time-fractional Fokker--Planck--Kolmogorov equations}\label{Section:Appl_3}

\subsection{Approximation of solutions  for distributed order  time-fractional Fok\-ker--Planck--Kolmogorov equations}
We are interested now in distributed order  time-fractional evolution equations of the form
\begin{align}\label{Appl:3:eq:fracFPK}
\mathcal{D}^\mu f(t,x)=Lf(t,x),
\end{align}
 where $\mathcal{D}^\mu$ is the distributed order fractional derivative with respect to the time variable $t$ and $L$ is the generator of a strongly continuous semigroup $(T_t)_{t\geq0}$ on some Banach space $(X,\|\cdot\|_X)$ of functions of the space variable $x$.  Equations of such type are called \emph{time-fractional Fokker--Planck--Kolmogorov equations} and arise in the framework of continuous time random walks (CTRWs) and fractional kinetic theory (\cite{MR0260036,MR757002,MR1809268,MR1937584}). 
As it is shown in papers \cite{MR2782245,MR2886388,MR3168478} (see also papers \cite{MR1656314,MR1874479,MR2442372,MR2179231,MR3049524} for the case $\mu=\delta_{\beta_0}$, $\beta_0\in(0,1)$),  such time-fractional Fokker--Planck--Kolmogorov equations are governing equations for stochastic processes which are  weak limits of certain sequences or triangular arrays of CTRWs. These limit processes are actually time-changed L\'{e}vy processes, where the time-change arises as the first hitting time of level $t>0$ (or, equivalently, as the inverse process) for a  mixture of independent stable subordinators with some  mixing measure $\mu$.

Recall that a process $(D^\beta_t)_{t\geq0}$ with $\beta\in(0,1)$ is  \emph{$\beta$-stable subordinator} if it is a one-dimensional L\'{e}vy proces with  almost surely non-decreasing paths  such that the corresponding Bernstein function is $f(s):=s^\beta$ (see, e.g.,  Section 3.9 of~\cite{MR1873235} for the definitions).   For a given finite Borel measure $\mu$ with $\supp\mu\in(0,1)$, consider the function $f^\mu$ given by
$$
f^\mu(s):=\intl_0^1 s^\beta\mu(d\beta), \quad s>0.
$$
It is a Bernstein function  given by (cf.  formula (3.246) of~\cite{MR1873235}):
$$
f^\mu(s)=\intl_{0+}^\infty \left(1-e^{-ts}\right)m(dt),\quad\text{where}\quad m(dt):=\left(\intl_0^1\frac{\beta t^{-\beta-1}}{\Gamma(1-\beta)}\mu(d\beta)\right)dt.
$$
Let $(D^\mu_t)_{t\geq0}$ be a subordinator corresponding to the Bernstein function $f^\mu$.   This process represents a mixture of independent stable subordinators with a mixing measure $\mu$. Define now the process $(E^\mu_t)_{t\geq0}$
by
$$
E^\mu_t:=\inf\left\{\tau\geq0\,:\,D^\mu_\tau>t \right\}.
$$
The process $(E^\mu_t)_{t\geq0}$ is the first hitting time of the level $t$ of the process $(D^\mu_\tau)_{\tau\geq0}$ or, equivalently, the inverse to $(D^\mu_t)_{t\geq0}$. This process $(E^\mu_t)_{t\geq0}$  is sometimes called \emph{inverse subordinator}. However,  note that it is not a Markov process. It is known that $(E^\mu_t)_{t\geq0}$ possesses a marginal density function $p^\mu(t,x)$, i.e. $\mathbb{P}(E^\mu_t\in B)=\int_B p^\mu(t,x)dx$ for all $B\in\mathcal{B}(\cR)$, and $p^\mu(t,x)=0$ for all $x<0$.  The marginal density function $p^\mu(t,x)$ has many nice properties (see Lemma~2.4 and Lemma~2.5 in \cite{MR2736349}; for the case  $\mu=\delta_{\beta_0}$, $\beta_0\in(0,1)$, see also \cite{MR2074812,MR2726092}).  In particular, $p^\mu\in C^\infty((0,\infty)\times(0,\infty))$. 
In the sequel, we need the following simple property of $p^\mu$:

\begin{lemma}\label{Appl:3:le}
For each $\eps>0$ and each $T>0$  there exists  $R_{T,\eps}>0$ such that for all $t\in[0,T]$ holds
\begin{align*}
\intl_{R_{T,\eps}}^\infty p^\mu(t,x)dx<\eps.
\end{align*}
\end{lemma}
 
\begin{proof}
Choose arbitrary $\eps>0$ and  $T>0$.  Consider $R>0$.  We have for all $t\in[0,T]$:
\begin{align*}
\intl_{R}^\infty p^\mu(t,x)dx&=\PP\left(E^\mu_t\geq R\right)=\PP\left(D^\mu_R\leq t\right)\\
&
\leq \PP\left(D^\mu_R\leq T\right)=\PP\left(E^\mu_T\geq R\right)=\intl_{R}^\infty p^\mu(T,x)dx\\
&
<\eps
\end{align*}
for  sufficiently large $R$ since $\int_0^\infty p^\mu(T,x)dx=1$.
\end{proof}

The marginal density function $p^\mu(t,x)$ allows to represent solutions of Cauchy problems for distributed order time-fractional evolution equations of the form \eqref{Appl:3:eq:fracFPK} in the following way (cf. Thm.~3.2 in \cite{MR2886388} and Thm.~4.2 in \cite{MR3168478}):

\begin{proposition}\label{Appl:3:prop}
Let $(X,\|\cdot\|_X)$ be a Banach space. Let $(L,\Dom(L))$ be the generator of a uniformly bounded\footnote{This means that $\|T_t\|\leq M$ for some $M>0$ and all $t\geq0$.}, strongly continuous semigroup $(T_t)_{t\geq0}$ on $X$. Let $f_0\in\Dom(L)$. Let $\mu$ be a finite Borel measure with $\supp\mu\in(0,1)$.  Then the family of linear operators $(\mathcal{T}_t)_{t\geq0}$ from $X$ into $X$ given by
\begin{align}\label{Appl:3:eq:Tt-subordinated}
\mathcal{T}_t\ffi:=\intl_0^\infty T_\tau\ffi\, p^\mu(t,\tau)\,d\tau,\quad\forall\,\ffi\in X,
\end{align} 
is uniformly bounded and strongly analytic in a sectorial region. Furthermore, the family $(\mathcal{T}_t)_{t\geq0}$ is strongly continuous  and the function $f(t):=\mathcal{T}_t f_0$ is a solution of the Cauchy problem
\begin{align}\label{Appl:3:eq:CP}
&\mathcal{D}^\mu f(t)=Lf(t),\quad t>0,\nonumber\\
&
f(0)=f_0.
\end{align}
\end{proposition}

This result shows that solutions of time-fractional evolution equations are a kind of subordination of  solutions of the corresponding time-non-fractional evolution equations with  respect to ``subordinators'' $(E^\mu_t)_{t\geq0}$. And  respectively,  if a time-non-fractional evolution equation is a governing equation for a  Markov process then the related time-fractional evolution equation is a governing equation for a (already non-markovian) process which is a ``subordination'', i.e.  a time-change of this Markov process by means of $(E^\mu_t)_{t\geq0}$. The non-Markovity of the resulting process corresponds  to the fact that the family $(\mathcal{T}_t)_{t\geq0}$ is not a semigroup any more.   Note also that some other types of  time-fractional evolution equations have a similar ``subordination-like'' structure of solutions (see \cite{MR2964432,MR2588003}). 

Assume now that the semigroup $(T_t)_{t\geq0}$ is not known explicitly but is already Chernoff approximated. We have no chances to construct Chernoff approximations for the family $(\mathcal{T}_t)_{t\geq0}$ which is not a semigroup.  Nevertheless, the following is true.

\begin{theorem}\label{Appl:3:thm}
Let $(X,\|\cdot\|_X)$ be a Banach space. Let $(L,\Dom(L))$ be the generator of a  strongly continuous contraction semigroup $(T_t)_{t\geq0}$ on $X$. Let the family $(F(t))_{t\geq0}$ of contractions on $X$ be strongly Borel measurable\footnote{I.e., for  all $\ffi\in X$, the mappings $t\mapsto\| F(t)\ffi\|_X$  from $[0,\infty)$ to $\cR$ are Borel measurable. This is fulfilled if, e.g., the family $(F(t))_{t\geq0}$ is strongly continuous. And if the family $(F(t))_{t\geq0}$ of contractions on $X$ is such that all the mappings $t\mapsto\| F(t)\ffi\|_X$, $\ffi\in X$, are Borel measurable, then  Bochner integrals in~\eqref{eq:fn} are well-defined and finite.} and Chernoff equivalent to $(T_t)_{t\geq0}$. Let $f_0\in\Dom(L)$. Let $\mu$ be a finite Borel measure with $\supp\mu\in(0,1)$ and the family $(\mathcal{T}_t)_{t\geq0}$ be given by formula~\eqref{Appl:3:eq:Tt-subordinated}. Let $f\,:\,[0,\infty)\to X$ be defined via $f(t):=\mathcal{T}_t f_0$.  For each $n\in\Nat$  define  the mappings  $f_n\,:\,[0,\infty)\to X$  by
\begin{align}\label{eq:fn}
f_n(t):=\intl_0^\infty F^n(\tau/n)f_0\,p^\mu(t,\tau)\,d\tau.
\end{align}
 Then it holds  locally uniformly with respect to  $t\geq0$ that
$$
\|f_n(t)-f(t)\|_X\to0,\quad n\to\infty.
$$
\end{theorem}

\begin{proof}
Take any $T>0$ and any $\eps>0$. Due to Lemma~\ref{Appl:3:le}, there exists $R_{T,\eps}>0$ such that 
$$
\intl_{R_{T,\eps}}^\infty p^\mu(t,\tau)d\tau<\eps
$$
for all $t\in[0,T]$. Then it holds for $t\in[0,T]$
\begin{align*}
\|f_n&(t)-f(t)\|_X=\left\|\intl_0^\infty F^n(\tau/n)f_0\,p^\mu(t,\tau)\,d\tau-\intl_0^\infty T_\tau f_0\, p^\mu(t,\tau)\,d\tau\right\|_X\\
&
\leq \intl_0^\infty \|T_\tau f_0-F^n(\tau/n) f_0\|_X\, p^\mu(t,\tau)\,d\tau\\
&
\leq \intl_0^{R_{T,\eps}} \|T_\tau f_0-F^n(\tau/n) f_0\|_X\, p^\mu(t,\tau)\,d\tau +\intl_{R_{T,\eps}}^\infty \|T_\tau f_0-F^n(\tau/n) f_0\|_X\, p^\mu(t,\tau)\,d\tau\\
&
\leq \supl_{\tau\in[0,R_{T,\eps}]} \|T_\tau f_0-F^n(\tau/n) f_0\|_X\intl_0^{R_{T,\eps}}p^\mu(t,\tau)\,d\tau+2\eps\|f_0\|_X\\
&
\to 2\eps\|f_0\|_X,\quad n\to\infty,
\end{align*}
due to the fact that  the convergence in the Chernof theorem is locally uniform with respect to the time variable. Since $\eps>0$ was chosen arbitrary, the statement follows.
\end{proof}

\begin{remark}
Consider a time-fractional Fokker--Planck--Kolmogorov equation of the form~\eqref{Appl:3:eq:fracFPK}. Assume that the semigroup $(T_t)_{t\geq0}$, whose generator $L$ stands in the right hand side of the equation, corresponds to a Markov process $(\xi(t))_{t\geq0}$. Then this time-fractional Fokker--Planck--Kolmogorov equation is a governing equation for the stochastic process $(\xi(E^\mu_t))_{t\geq0}$  which is the  time-change  of $(\xi(t))_{t\geq0}$ by means of  the inverse subordinator $(E^\mu_t)_{t\geq0}$. And  the function 
\begin{align}\label{Appl:3:eq:f-FKF}
f(t,x):=\E\left[f_0\left(\xi\left(E^\mu_t\right)\right)\,\,|\,\,\xi(E^\mu_0)=x\right]
\end{align}
solves the Cauchy problem~\eqref{Appl:3:eq:CP} (cf. Theorem~3.6 in \cite{MR2886388}, see also \cite{MR1874479,MR2074812}). 
Since the process $(\xi(E^\mu_t))_{t\geq0}$ is not Markov,  its marginal density function (together with the initial distribution) does not determine all finite dimensional distributions of the process.  And there exist many different processes with the same marginal density function. Hence there exist many other stochastic representations for the function $f(t,x)$  in formula~\eqref{Appl:3:eq:f-FKF} (see, e.g., Thm~3.3 in \cite{MR3413862}, Cor.~3.4 in~\cite{MR2491905} and results of  \cite{MR2489164}). Furthermore, the considered  time-fractional Fokker--Planck--Kolmogorov equations (with $\mu=\delta_{\beta_0}$, $\beta_0\in(0,1)$) are related to some time-non-fractional evolution equations of hihger order (see, e.g., \cite{MR2491905,MR2965747}). Therefore, the approximations  $f_n$ constructed in Theorem~\ref{Appl:3:thm}   can be used simultaneousely to  approximate  path integrals appearing in different sto\-chas\-tic representations of the same function $f(t,x)$ and to  approximate  solutions of corresponding time-non-fractional evolution equations of hihger order.
\end{remark}

\begin{remark}\label{Appl:3:rem:CompareWithChapter4}
Obviousely, approximations $f_n$ similar to those of Theorem~\ref{Appl:3:thm}  can be constructed also for subordinate semigroups discussed in Section~3.1  of \cite{ChApprSubSem}. Na\-me\-ly, assume that
a semigroup $(T^f_t)_{t\geq 0}$ is subordinate   to a given semigroup $(T_t)_{t\geq 0}$ on a Banach space $(X,\|\cdot\|_X)$ with respect to a given convolution semigroup $(\eta_t)_{t\geq 0}$ associated to a Bernstein function $f$, i.e.  $T^f_t\ffi=\int_0^\infty T_s\ffi\,\eta_t(ds)$ for all $\ffi\in X$; the convolution semigroup $(\eta_t)_{t\geq 0}$ is known explicitly; and a given  family $(F(t))_{t\geq 0}$ (of contractions) is Chernoff equivalent to $(T_t)_{t\geq 0}$ and strongly Borel measurable. Then, similarly to the proof of Theorem~\ref{Appl:3:thm}, one shows that the functions $f_n(t)$,
\begin{align}\label{Appl:3:eq:f_n-Subord}
f_n(t):=\intl_0^\infty F^n(s/n)f_0\,\eta_t(ds),
\end{align}
approximate the function $T^f_t f_0$ in the norm $\|\cdot\|_X$ locally uniformly with respect to $t\geq0$ for all $f_0\in \Dom(L)$. Note that such approximations $f_n$ are much simpler than Chernoff approximations  based on the family $(\mathcal{F}(t))_{t\geq 0}$ constructed in Theorem~3.1 of \cite{ChApprSubSem}.  However,  the Chernoff approximations, based on the  families $(\mathcal{F}(t))_{t\geq 0}$ which are presented in Theorem~3.1 of \cite{ChApprSubSem}, can be used as a building block for further purposes. First,  families $(\mathcal{F}(t))_{t\geq 0}$ can be used to obtain Chernoff approximations for  semigroups, constructed by several iterative procedures of subordination, killing of an underlying process upon leaving a given domain, additive and multiplicative perturbations (of generators)  of some original semigroups. Second,  such Chernoff approximations, in turn, can be used to obtain approximations for solutions of the corresponding time-fractional evolution equations. Whereas the approximations $f_n$  in formula~\eqref{Appl:3:eq:f_n-Subord}  can be used only to approximate  $T^f_tf_0$ and not for further purposes.  
\end{remark}

 \subsection{Feynman formula solving the Cauchy problem for a class of   time-fractional diffusion equations}\label{Appl:3:example}

Let us  return to  the situation described in  Corollary~\ref{cor:example:CDP-BGS}. 
So, let $(T_t)_{t\geq0}$ be  a Feller semigroup\footnote{Hence all $T_t$ are contractions.}  whose generator $(L,\Dom(L))$ is given for all $\ffi\in C^2_0(\cRd)$ by formula~\eqref{1:3:eq:L_ABC}. Let coefficients $A$, $B$, $C$ be bounded and continuous and \eqref{eq:CP:uniformEllipt} hold.  Let $C^{2,\alpha}_c(\cRd)$ be a core for $(L,\Dom(L))$.
 Consider the  family $(F(t))_{t\geq0}$ of bounded linear operators on $X$ given by formula~\eqref{1:eq:F^ABC+}. Under the  assumptions above, this family is strongly continuous and Chernoff equivalent to the semigroup $(T_t)_{t\geq0}$. And all operators $F(t)$ are contractions since $C(x)\geq0$ for all $x\in\cRd$.
Let $f_0\in\Dom(L)$. Due to Proposition~\ref{Appl:3:prop}, the function $$f(t,x):=\intl_0^\infty  T_\tau f_0(x)\, p^{\mu}(t,\tau)\,d\tau$$ solves the Cauchy problem for the distributed order  time-fractional diffusion equation 
\begin{align}\label{eq:fracDiff-CP}
\mathcal{D}^\mu f(t,x)=-C(x)f(t,x) - B(x)\cdot\nabla  f(t,x)+ \tr(A(x)\Hess f(t,x)) 
\end{align}
with the initial condition $f(0,x)=f_0(x)$. Due to Theorem~\ref{Appl:3:thm}, the following functions $f_n(t,x)$ approximate the solution  $f(t,x)$ in supremum norm with respect to $x\in\cRd$  uniformly with respect to $t\in(0,t^*]$ for all $t^*>0$ as $n\to\infty$:
\begin{align*}
f_n(t,x_0):=\intl_{0}^{\infty}&
\intl_{\cR^{nd}} e^{-\frac{\tau}{n}\suml_{j=1}^n \left(C(x_{j-1})+\frac14\left|A^{-1/2}(x_{j-1})B(x_{j-1})\right|^2   \right) } e^{\frac12\suml_{j=1}^n   A^{-1}(x_{j-1})B(x_{j-1})\cdot(x_{j-1}-x_{j})} 
\nonumber\\
&
 \times
 p_A(\tau/n, x_0,x_1) \cdots p_A(\tau/n, x_{n-1},x_n) f_0(x_n)p^{\mu}(t,\tau)dx_1\ldots dx_{n}d\tau.
\end{align*}
Since for each $x_0\in\cRd$ the solution $f(t,x_0)$ is the limit of $f_n(t,x_0)$, i.e. the limit of  $(n+1)-$fold  iterated  integrals  as $n\to\infty$, the approximations  $f_n(t,x_0)$ provide us  just a Feynman formula for $f(t,x_0)$. Namely, the following statement holds.

\begin{proposition}
Under assumptions of this Subsection, the function $f(t,x_0)$, given by the Feynman formula~\eqref{FF:fracCP} below, solves the Cauchy problem for the distributed order time-fractional diffusion equation~\eqref{eq:fracDiff-CP} with the initial condition $f_0$.
\begin{align}\label{FF:fracCP}
&f(t,x_0)\nonumber\\
&=\liml_{n\to\infty}\intl_{0}^{\infty}
\intl_{\cR^{nd}} e^{-\frac{\tau}{n}\suml_{j=1}^n \left(C(x_{j-1})+\frac14\left|A^{-1/2}(x_{j-1})B(x_{j-1})\right|^2   \right) } e^{\frac12\suml_{j=1}^n   A^{-1}(x_{j-1})B(x_{j-1})\cdot(x_{j-1}-x_{j})} 
\nonumber\\
&
\quad\quad\quad\quad\quad\quad \times
  p_A(\tau/n, x_0,x_1) \cdots p_A(\tau/n, x_{n-1},x_n) f_0(x_n)p^{\mu}(t,\tau)dx_1\ldots dx_{n}d\tau,
\end{align}
where $p_A$ is given by~\eqref{1:eq:3:p_A}. And the convergence is  uniform with respect to $x_0\in\cRd$ and with respect to $t\in(0,t^*]$ for all $t^*>0$.
\end{proposition}

In particular, consider  a $1/2$-stable inverse subordinator $(E^{1/2}_t)_{t\geq0}$. Its marginal probability density is known explicitly (see Cor.~3.1 in \cite{MR2074812} and the discussion after  Lemma~3 in \cite{MR2726092})
\begin{align*}
p^{1/2}(t,\tau)=\frac{1}{\sqrt{\pi t}}e^{-\frac{\tau^2}{4t}}.
\end{align*}
Therefore, in the case when $\mathcal{D}^\mu $ is the Caputo derivative of $1/2$-th order, the  function $f(t,x_0)$,
represented (uniformly with respect to $x_0\in\cRd$ and with respect to $t\in(0,t^*]$ for all $t^*>0$) by the following Feynman formula~\eqref{FF:fracCP-1/2},  solves   the Cauchy problem for the   time-fractional diffusion equation~\eqref{eq:fracDiff-CP}  with initial condition $f_0$:
\begin{align}\label{FF:fracCP-1/2}
&f(t,x_0)\nonumber\\
&=\liml_{n\to\infty}\intl_{0}^{\infty}
\intl_{\cR^{nd}} e^{-\frac{\tau}{n}\suml_{j=1}^n \left(C(x_{j-1})+\frac14\left|A^{-1/2}(x_{j-1})B(x_{j-1})\right|^2   \right) } e^{\frac12\suml_{j=1}^n   A^{-1}(x_{j-1})B(x_{j-1})\cdot(x_{j-1}-x_{j})} 
\nonumber\\
&
\quad\quad\quad\quad\quad\quad \times
  p_A(\tau/n, x_0,x_1) \cdots p_A(\tau/n, x_{n-1},x_n) f_0(x_n)\frac{1}{\sqrt{\pi t}}e^{-\frac{\tau^2}{4t}}dx_1\ldots dx_{n}d\tau.
\end{align}

Analogous results hold true for distributed order time-fractional Fokker--Planck--Kolmogorov equations with non-local operators $L$ considered  in Subsection~\ref{Subsection:CDP-nonLoc} and in \cite{ChApprSubSem,MR2999096}.

\subsection{Feynman formula solving the Cauchy--Dirichlet problem  for a class of time-fractional diffusion equations}\label{Appl:3:example-CDP}

We keep on working in the situation of Subsection~\ref{Appl:3:example}. Let additionally   $(T_t)_{t\geq0}$ be   doubly Feller, $A$, $B$, $C$ be of the class $C^{2,\alpha}_b(\cRd)$ for some $\alpha\in(0,1)$. Let $G\subset\cRd$ be a  bounded domain with the boundary $\pd G$ of the class $C^{4,\alpha}$ for  some $\alpha\in(0,1)$.  Then  we have by Theorem~\ref{CDP:2:thm} and Remark~\ref{CDP:2:rem:L_ABC-i}  for the corresponding  semigroup $(T^o_t)_{t\geq0}$ on $Y$ 
\begin{align*}
T^o_t\varphi(x) = \liml_{n\to\infty}\big(\big(F_o(t/n)\big)^n\varphi\big)(x) \quad \mbox{for all } \,\ffi\in C_0(G),
\end{align*}
uniformly in $x \in G$ and locally uniformly in $t \in [0,\infty)$. Here the family $(F_o(t))_{t\geq0}$ has been constructed from the family $(F(t))_{t\geq0}$ given in~\eqref{1:eq:F^ABC+} by the formula~\eqref{1:eq:4:Fo}. 
Let  $f_0\in\Dom(L_o)$. Due to Proposition~\ref{Appl:3:prop} and Remark~\ref{rem:ACP-CDP}, the function 
\begin{align}\label{eq:f-fractCDP}
f(t,x):=\intl_0^\infty  T^o_\tau f_0(x)\, p^{\mu}(t,\tau)\,d\tau
\end{align}
 solves the following Cauchy--Dirichlet problem for the distributed order time-fractional diffusion equation
\begin{align}\label{eq:fractCDP}
&\mathcal{D}^\mu f(t,x)=-C(x)f(t,x) - B(x)\cdot\nabla  f(t,x)+ \tr(A(x)\Hess f(t,x)),\quad t>0,\,\,x\in G,\\
&
f(0,x)=f_0(x),\quad x\in G,\nonumber\\
&
f(t,x)=0,\quad t>0,\,\,x\in\pd G. \nonumber
\end{align}
 Due to Theorem~\ref{Appl:3:thm}, the following functions $f_n(t,x)$ approximate the solution  $f(t,x)$ in supremum norm with respect to $x\in G$  uniformly with respect to  $t\in(0,t^*]$ for all $t^*>0$:
\begin{align*}
&f_n(t,x_0):=\intl_0^\infty  F_o^n(\tau/n) f_0(x_0)\, p^{\mu}(t,\tau)\,d\tau\\
&
=\intl_{0}^{\infty}
\intl_{\cR^{nd}} e^{-\frac{\tau}{n}\suml_{j=1}^n \left(C(x_{j-1})+\frac14\left|A^{-1/2}(x_{j-1})B(x_{j-1})\right|^2   \right) } e^{\frac12\suml_{j=1}^n   A^{-1}(x_{j-1})B(x_{j-1})\cdot(x_{j-1}-x_{j})}\mathcal{E}(f_0)(x_n) 
\nonumber\\
&
 \times
\bigg(\prod^{n}_{j=1}\phi_{s(\tau/n)}(x_{j-1})\bigg)  p_A(\tau/n, x_0,x_1) \cdots p_A(\tau/n, x_{n-1},x_n) p^{\mu}(t,\tau)\,dx_1\ldots dx_{n}d\tau.
\end{align*}
For each $\tau\in[0,\infty)$, each $x_0\in G$ and each $n\in\Nat$, let us define $\Theta^{f_0}_n(\tau,x_0)$ by
\begin{align*}
\Theta^{f_0}_n(\tau,x_0):=\intl_{G^n}& e^{-\frac{\tau}{n}\suml_{j=1}^n \left(C(x_{j-1})+\frac14\left|A^{-1/2}(x_{j-1})B(x_{j-1})\right|^2   \right) } e^{\frac12\suml_{j=1}^n   A^{-1}(x_{j-1})B(x_{j-1})\cdot(x_{j-1}-x_{j})}\\
&
\times f_0(x_n)   p_A(\tau/n, x_0,x_1) \cdots p_A(\tau/n, x_{n-1},x_n) dx_1\ldots dx_{n}.
\end{align*}
Then we have $\sup_{x_0\in G}|\Theta^{f_0}_n(\tau,x_0)|\leq\|f_0\|_Y$ for all $\tau\in[0,\infty)$ and $n\in\Nat$. Let us show that the functions 
$$
g_n(t,x_0):=\intl_0^\infty \Theta^{f_0}_n(\tau,x_0) p^{\mu}(t,\tau)d\tau
$$
approximate the function $f(t,x_0)$  in formula~\eqref{eq:f-fractCDP}  solving the Cauchy--Dirichlet problem~\eqref{eq:fractCDP} as $n\to\infty$  locally uniformly with respect to $x_0\in G$ and uniformly  with respect to $t\in(0,t^*]$ for all $t^*>0$. So, fix any $\eps>0$, $T>0$ and a compact $K\subset G$. Let $x_0\in K$ and $t\in[0,T]$. Due to Lemma~\ref{Appl:3:le}, there exists $R_{T,\eps}>0$ such that 
$
\intl_{R_{T,\eps}}^\infty p^{\mu}(t,\tau)d\tau<\eps
$
for all $t\in[0,T]$. Then, similarly to the proof of Theorem~\ref{Appl:3:thm},
\begin{align*}
&\left|g_n(t,x_0)-f(t,x_0)\right|\leq \left|g_n(t,x_0)-f_n(t,x_0)\right|+\left|f_n(t,x_0)-f(t,x_0)\right|\\
&\leq \intl_0^{R_{T,\eps}}\|\Theta^{f_0}_n(\tau,\cdot)- F^n_o(\tau/n) f_0\|_{C(K)} p^{\mu}(t,\tau)d\tau+2\eps\|f_0\|_Y+\|f_n(t,\cdot)-f(t,\cdot)\|_Y\\
&
\leq \supl_{\tau\in[0,R_{T,\eps}]}\|\Theta^{f_0}_n(\tau,\cdot)- F^n_o(\tau/n) f_0\|_{C(K)}+2\eps\|f_0\|_Y+\|f_n(t,\cdot)-f(t,\cdot)\|_Y\\
&\to 2\eps\|f_0\|_Y\quad\text{ as }\, n\to\infty,
\end{align*} 
since the convergence of $\Theta^{f_0}_n(\tau,x_0)$ to $F^n_o(\tau/n) f_0(x_0)$ is locally uniform with respect to $x_0\in G$ and uniform with respect  to $t\in(0,t^*]$ for all $t^*>0$ (cf. Corollary~\ref{cor:example:CDP-BGS}) and due to Theorem~\ref{Appl:3:thm}.  Therefore, since $\eps>0$ is arbitrary, the following statement is proved.
\begin{proposition}
Let $(T_t)_{t\geq0}$ be a doubly Feller semigroup on $X=C_0(\cRd)$ whose generator $(L,\Dom(L))$ is given for all $\ffi\in C^2_0(\cRd)$ by the formula~\eqref{1:3:eq:L_ABC}. Let the coefficients $A$, $B$, $C$ in~\eqref{1:3:eq:L_ABC} be  of the class $C^{2,\alpha}_b(\cRd)$ for some $\alpha\in(0,1)$. Let there exist $a_0$, $A_0\in\cR$ such that \eqref{eq:CP:uniformEllipt} holds. Assume that the coefficients $A$, $B$, $C$  are such that the set $C^{2,\alpha}_c(\cRd)$ is a core for the generator $L$ in $X$.  Let $G\subset\cRd$ be a bounded domain with the boundary $\pd G$ of the class $C^{4,\alpha}$. Let $(T^o_t)_{t\geq0}$ be the corresponding Feller semigroup on $Y=C_0(G)$ with the generator $(L_o,\Dom(L_o))$. Let $f_0\in\Dom(L_o)$. Let $\mu$ be a finite Borel measure with $\supp\mu\in(0,1)$.  Then the function $f(t,x_0)$, which is given for all $t>0$ and all $x_0\in G$ by the Feynman formula~\eqref{FF:fracCDP} below, solves the Cauchy--Dirichlet problem~\eqref{eq:fractCDP}. And the convergence in the Feynman formula~\eqref{FF:fracCDP} is locally uniform with respect to $x_0\in G$ and uniform with respect to $t\in(0,t^*]$ for all $t^*>0$:
 \begin{align}\label{FF:fracCDP}
&f(t,x_0)\nonumber\\
&=\liml_{n\to\infty}\intl_{0}^{\infty}
\intl_{G^{n}} e^{-\frac{\tau}{n}\suml_{j=1}^n \left(C(x_{j-1})+\frac14\left|A^{-1/2}(x_{j-1})B(x_{j-1})\right|^2   \right) } e^{\frac12\suml_{j=1}^n   A^{-1}(x_{j-1})B(x_{j-1})\cdot(x_{j-1}-x_{j})} 
\nonumber\\
&
\quad\quad\quad\quad\quad\quad \times
  p_A(\tau/n, x_0,x_1) \cdots p_A(\tau/n, x_{n-1},x_n) f_0(x_n)p^{\mu}(t,\tau)dx_1\ldots dx_{n}d\tau,
\end{align}
where $p_A$ is given by~\eqref{1:eq:3:p_A}.
\end{proposition}

Analogous results hold true for distributed order time-fractional Fokker--Planck--Kolmogorov equations with non-local operators $L$ considered  in Subsection~\ref{Subsection:CDP-nonLoc}. Other representations for solutions of some distributed order time-fractional Fokker--Planck--Kolmogorov equations in bounded domains  can be found, e.g., in  \cite{MR2776466,MR2921690}.  

\section*{Acknowledgements}
I would like to thank Christian Bender for the encouragement and support of my researches;
 Alessandra Lunardi for her remarks (they are presented in Remark~\ref{rem:Lunardi}) on properties of the Laplace operator with Dirichlet boundary conditions;    Krzysztof Bogdan for communicating me the reference \cite{BLM} and Panki Kim for  remarks on killed Feller processes.

%

\def\cprime{$'$}

\end{document}